\documentclass[11pt]{amsart}
\usepackage{graphicx,amsfonts,amssymb,amsmath,amsthm,url,
  verbatim,amscd, fullpage}
  \usepackage{color}
\usepackage[usenames,dvipsnames]{xcolor}
\usepackage[normalem]{ulem}
\usepackage{pdfsync}
\usepackage{booktabs}
\usepackage[hyperfootnotes=false, colorlinks, citecolor=RoyalBlue,
urlcolor=blue, linkcolor=blue ]{hyperref}

\newtheorem{lemma}{Lemma}[section]
\newtheorem{theorem}[lemma]{Theorem}

\newtheorem{corollary}[lemma]{Corollary}
\newtheorem{proposition}[lemma]{Proposition}

\theoremstyle{remark}
\newtheorem{definition}[lemma]{Definition}
\newtheorem{remark}[lemma]{Remark}

\newtheorem{example}[lemma]{Example}

\allowdisplaybreaks

\renewcommand{\H}{\mathbb H}
\newcommand{\A}{{\mathbb A}}
\newcommand{\Aa}{{\mathcal A}}
\newcommand{\m}{{\rm max}}
\newcommand{\cond}{{\rm cond}}
\renewcommand{\a}{{\mathbf a}}
\newcommand{\f}{{\mathbf f}}
\renewcommand{\c}{{\mathbf c}}
\newcommand{\q}{{\mathfrak q}}
\newcommand{\F}{{\mathcal F}}
\newcommand{\J}{{\mathcal J}}
\newcommand{\Q}{{\mathbb Q}}
\newcommand{\Z}{{\mathbb Z}}

\renewcommand{\O}{{\mathcal O}}
\newcommand{\G}{{\mathcal G}}
\renewcommand{\L}{{\mathcal L}}

\newcommand{\R}{{\mathbb R}}
\newcommand{\C}{{\mathbb C}}
\newcommand{\bs}{\backslash}

\newcommand{\p}{\mathfrak p}

\renewcommand{\S}{\mathfrak S}
\renewcommand{\P}{\mathcal P}
\newcommand{\OF}{{\mathfrak o}}
\newcommand{\GL}{{\rm GL}}

\newcommand{\SL}{{\rm SL}}
\newcommand{\SO}{{\rm SO}}

\newcommand{\ram}{{\bf S}}
\newcommand{\ur}{{\bf ur}}

\newcommand{\Tr}{{\rm tr}}
\newcommand{\tr}{{\rm Tr}}
\newcommand{\nr}{{\rm nr}}
\newcommand{\Ad}{{\rm Ad}}

\newcommand{\disc}{{\rm disc}}

\newcommand{\mat}[4]{{\setlength{\arraycolsep}{0.5mm}\left(
      \begin{array}{cc}#1&#2\\#3&#4\end{array}\right)}}

\newcommand{\AI}{{\mathcal{AI}}}

\def\SL{\operatorname{SL}}
\def\vol{\operatorname{vol}}
\def\GL{\operatorname{GL}}

\renewcommand{\Im}{\mathrm{Im}}
\def\eps{\epsilon}

\newcommand{\Phii}{\Phi_\pi^{(i)}(a,m)}
\newcommand{\zxz}[4]{\begin{pmatrix} #1 & #2 \\ #3 & #4 \end{pmatrix}}

\newcommand{\E}{\ensuremath{E}}
\newcommand{\D}{\ensuremath{{\Delta}}}
\newcommand{\Vol}{\text{Vol}}
\newcommand{\AL}{\ensuremath{\alpha}}

\begin{document}

\bibliographystyle{plain}

\title[{Sup-norm in the level aspect for compact arithmetic surfaces, II}]{Sup-norms of eigenfunctions in the level aspect for compact arithmetic surfaces, II: newforms and subconvexity}

\author{Yueke Hu}
\address{Yau Mathematical Sciences Center\\
  Tsinghua University\\
  Beijing 100084\\
  China}
\email{yhumath@mail.tsinghua.edu.cn}

\author{Abhishek Saha}
\address{School of Mathematical Sciences\\
  Queen Mary University of London\\
  London E1 4NS\\
  UK}
  \email{abhishek.saha@qmul.ac.uk}

\begin{abstract}We improve upon the local bound in the depth aspect for sup-norms of newforms on $D^\times$ where $D$ is an indefinite quaternion division algebra over $\Q$. Our sup-norm bound implies a depth-aspect subconvexity bound for $L(1/2, f \times \theta_\chi)$, where $f$ is a (varying) newform on $D^\times$ of level $p^n$, and $\theta_\chi$ is an (essentially fixed) automorphic form on $\GL_2$ obtained as the theta lift of a Hecke character $\chi$ on a quadratic field.

  For the proof, we augment the amplification method with  a novel filtration argument and a recent counting result proved by the second-named author to reduce to showing strong quantitative decay of matrix coefficients of local newvectors along compact subsets, which we establish via $p$-adic stationary
phase analysis. Furthermore, we prove a general upper bound in the level aspect for sup-norms of automorphic forms belonging to \emph{any} family whose associated matrix coefficients
have such a decay property.




\end{abstract}


\maketitle

\section{Introduction}\label{s:introduction}
Let $D$ be an indefinite quaternion  algebra over $\Q$. For any integer $N$ coprime to the discriminant of $D$, let $\Gamma^D_0(N)\subset \SL_2(\R)$ denote the congruence subgroup\footnote{This subgroup is well-defined up to conjugation in $\SL_2(\R)$ and we fix a choice once and for all.} corresponding to the norm 1 units of an Eichler order of level $N$ inside $D$. There has been a lot of work on bounding the sup norm $\|f \|_\infty$ of a Hecke-Maass \emph{newform} of weight 0 and Laplace eigenvalue $\lambda$ on $\Gamma^D_0(N) \bs \H$, where $f$ is $L^2$-normalized with respect to the measure that gives volume 1 to $\Gamma^D_0(N) \bs \H$. (For simplicity, we only discuss the case of newforms with trivial character in the introduction.)

The pioneering work here  is due to Iwaniec and Sarnak \cite{iwan-sar-85}, who proved the eigenvalue aspect bound\footnote{All implied constants in this paper will depend on $D$ without explicit mention.} $\|f\|_\infty \ll_\eps \lambda^{5/24 + \eps}$ in the case $N=1$. For the \emph{level} aspect analogue of this problem, the goal is to bound $\|f\|_\infty$ in terms of  $N$, with the dependance on $\lambda$ suppressed. It will be convenient to use the notation  $N_1$ to denote the smallest integer such that $N|N_1^2$. Clearly $\sqrt{N} \le N_1 \le N$. Note that $N_1$ equals $N$ if $N$ is squarefree while $N_1$ is around $\sqrt{N}$ when all the prime factors of $N$ divide it to a high power. To show the rapid progress in the level aspect version of the sup-norm problem for newforms on $D$, we quote the results proved so far in this direction\footnote{We do not attempt to survey other sup-norm results, such as the various recent works  concerning lower bounds,  hybrid bounds, holomorphic forms, general multiplier systems, general number fields, higher rank groups, exotic vectors at the ramified places, function field analogues, and so on and so forth. We refer the reader to the introductions of \cite{blomer-harcos-milicevic-maga, saha-sup-minimal} for brief discussions of some of these related results.}.

\subsubsection*{The case $D = M_2(\Q)$} The ``trivial bound" (which is not completely trivial, since one has to be careful about behaviour near cusps) is $\|f\|_\infty \ll_{\lambda, \epsilon} N^{\frac12+ \epsilon}$.  The following bounds were proved in rapid succession: \begin{itemize}
\item $\|f\|_\infty \ll_{\lambda, \epsilon} N^{\frac12 -\frac{25}{914} + \epsilon}$ for squarefree $N$ (Blomer and Holowinsky~\cite{blomer-holowinsky}, 2010);

\item  $\|f\|_\infty \ll_{\lambda, \epsilon} N^{\frac12-\frac{1}{22} + \epsilon}$ for squarefree $N$ (Templier~\cite{templier-sup}, 2010);

  \item  $\|f\|_\infty \ll_{\lambda, \epsilon} N^{\frac12-\frac{1}{20} + \epsilon}$ for squarefree $N$ (Helfgott--Ricotta, unpublished);

 \item $\|f\|_\infty \ll_{\lambda, \epsilon} N^{\frac12-\frac{1}{12} + \epsilon}$ for squarefree $N$ (Harcos and Templier~\cite{harcos-templier-1}, 2012);

  \item $\|f\|_\infty \ll_{\lambda, \epsilon} N^{\frac{1}{3} + \epsilon}$ for squarefree $N$ (Harcos and Templier~\cite{harcos-templier-2}, 2013);

          \item $\|f\|_\infty \ll_{\lambda, \epsilon} N^{\frac16 + \eps}
  N_1^{\frac16}$ for general $N$ (Saha~\cite{saha-sup-level-hybrid}, 2017).
      \end{itemize}

\subsubsection*{The case $D$ a division algebra} The ``trivial bound" is again $\|f\|_\infty \ll_{\lambda, \epsilon} N^{\frac12+ \epsilon}$.  The following improved bounds have been proved so far:

\begin{itemize}

\item  $\|f\|_\infty \ll_{\lambda, \epsilon} N^{\frac12-\frac{1}{24} + \epsilon}$ for general $N$ (Templier~\cite{templier-sup}, 2010);

          \item $\|f\|_\infty \ll_{\lambda, \epsilon}
  N_1^{\frac12+\eps}$ for general $N$ (Marshall~\cite{marsh15}, 2016);

      \item $\|f\|_\infty \ll_{\lambda, \epsilon}
  N^{\frac{1}{24}+\eps}N_1^{\frac12 - \frac{1}{12}}$ for general $N$ (Saha~\cite{saha-sup-minimal}, 2020);

      \end{itemize}

\medskip

Our main focus  in this paper is on a natural subcase of the level aspect, known as the \emph{depth} aspect, where one takes $N=p^n$ with $p$ a\emph{ fixed} prime and  $n$ varying. In this aspect, the best currently known bound for the sup-norm  is \begin{equation}\label{e:depthlocal}\|f\|_\infty \ll_{\lambda, p, \eps} p^{(n/4)(1+\eps)},\end{equation}as is clear from the list of previous results above; indeed, the bound \eqref{e:depthlocal} in the case $D=M_2(\Q)$ follows from work of the second-named author \cite{saha-sup-level-hybrid}  and  in the case when $D$ is a division algebra follows from work of Marshall \cite{marsh15}. More pertinently, the bound \eqref{e:depthlocal} coincides with the level aspect \emph{local} bound (which is stronger than the trivial bound\footnote{The level aspect local bound  is the immediate bound emerging from the adelic pre-trace formula where the local test function at each ramified prime is chosen to be essentially best possible. For a detailed discussion about  local bounds in a more general context, and its relationship with the trivial bound, see Section 1.4 of \cite{saha-sup-minimal}.}) which states in general that \begin{equation}\label{e:local}\|f\|_\infty \ll_{\lambda, \eps} N_1^{1/2 + \eps}.\end{equation} When we restrict ourselves to the depth aspect, we have $N_1 \asymp \sqrt{N}$ as $N=p^n \rightarrow \infty$ and so \eqref{e:depthlocal} is essentially \emph{equivalent} to \eqref{e:local} in this aspect. In contrast,  for squarefree levels $N$, we have $N_1=N$, and  the best currently known bounds in that case, due to Harcos-Templier \cite{harcos-templier-2} for $D=M_2(\Q)$ and Templier \cite{templier-sup} for $D$ a division algebra, successfully \emph{beat} the local bound by a positive power of $N$ as evidenced from the list of previous results quoted earlier. However, despite considerable recent activity on the sup-norm problem, the local bound in the depth aspect for newforms  has not been improved upon so far. We refer the reader to the end of the introduction of \cite{saha-sup-minimal} for a brief discussion why the usual methods are not sufficient to beat the local bound in this case.

In this paper, we improve upon \eqref{e:depthlocal} for the first time. For this, we introduce a new technique to attack the sup-norm problem which relies on quantifying the \emph{decay of local matrix coefficients} at the ramified primes along a filtration of compact subsets. To avoid dealing with behaviour at the cusps and Whittaker expansions, we restrict ourselves here to the case of $D$ a division algebra (though we have no doubt that our results can be extended to the case of $\GL_2$ with some additional technical work). We prove the following result.

\medskip

\textbf{Theorem A.} (see Corollary \ref{c:mainnewforms}) \emph{Let $D$ be a fixed indefinite quaternion division algebra over $\Q$ and $p$ be an odd prime coprime to the discriminant of $D$.  Then, for any $L^2$-normalized Maass newform $f$ of Laplace eigenvalue $\lambda$ on $\Gamma^D_0(p^n) \bs \H$, we have $$\|f\|_\infty \ll_{\lambda, p, \eps} p^{n\left(\frac{5}{24}+\eps\right)}.$$}

\begin{remark}Corollary \ref{c:mainnewforms} of this paper is more general than Theorem A in that it allows for general \emph{composite} levels (and the implied constant is polynomial in the product of primes dividing the level). Corollary \ref{c:mainnewforms} also includes the case of holomorphic forms $f$. Corollary \ref{c:mainnewforms}  is itself a very special case of the master theorem of this paper, Theorem \ref{t:maingen}, which applies to any family of automorphic forms on $D^\times(\A)$ satisfying certain hypotheses on decay of matrix coefficients.\end{remark}

We will explain the main ideas behind the proof of Theorem A later in this introduction, but first, let us describe an interesting application of this theorem to the \emph{subconvexity problem} for central $L$-values. The key idea, going back to Sarnak (see the nice exposition in Section 4 of \cite{sarnak93}), is that the conjectured strongest bounds for the sup-norms of automorphic forms often imply the Lindel\"of hypothesis in certain aspects for their associated $L$-functions. This leads to the question of whether one can use non-trivial sup-norm bounds to deduce subconvex bounds for $L$-functions. In this context, Iwaniec and Sarnak pointed out (see Remark D of \cite{iwan-sar-85}) that  their sup-norm result leads via Eisenstein series to a $t$-aspect subconvexity result for the Riemann zeta function. In fact, sup-norm bounds for Eisenstein series proved in \cite{young18} and \cite{blomer_2018} directly imply subconvexity bounds in the $t$-aspect for the Dedekind $L$-functions of imaginary quadratic fields (this follows by considering the values taken by the Eisenstein series at CM points). Very recently, uniform sup-norm bounds (with a dependence on the point of evaluation)  have been used in \cite{asbjorn} to prove hybrid subconvex bounds in the $t$ and $m$ aspects for  $L$-functions of ideal class characters of quadratic fields of discriminant $m$.

However, the above mentioned subconvexity results are only for $\GL_1$ $L$-functions and use  sup-norm bounds in the eigenvalue-aspect. Regarding the level aspect sup-norm problem for cusp forms on the upper-half plane and its connection to the subconvexity problem, see the discussion on page 647 of \cite{blomer-holowinsky}, which  points out that to prove \emph{any} level-aspect subconvex bound for an associated $L$-function by directly substituting a sup-norm bound into a period formula typically requires very strong\footnote{This is not surprising for at least two reasons: (a) sup-norm bounds hold for the whole space while period formulas only involve the values at a certain set of points or a submanifold, (b) substituting a sup-norm bound onto a period formula cannot detect any additional cancellation in the integral or sum involved in the formula.} sup-norm bounds. \emph{In particular, for level-aspect subconvexity, one needs to beat the exponent $1/4$ for the sup-norm problem}. Theorem A represents the first result that achieves this. Therefore, in this paper, we are finally able to carry out this strategy to deduce a depth-aspect subconvex bound from Theorem A. We give below a special case of the result we are able to obtain.

\medskip

\textbf{Theorem B.} (see Theorem \ref{t:sub}, Corollary \ref{c:sub} and Corollary \ref{c:gl2sub})  \emph{Let $D$ be a fixed indefinite quaternion division algebra over $\Q$ and let $p$ be an odd prime coprime to the discriminant of $D$. Let $K$ be a quadratic number field such that $p$ splits in $K$ and all primes dividing the discriminant of $D$ are inert in $K$. Let $\chi$ be a Hecke character of $K$ such that $\chi |_{\A^\times} =  1$ and such that the ramification set of $\chi$ does not intersect the places above $\disc(D)$. Let $\theta_\chi$ be the automorphic form\footnote{The condition $\chi |_{\A^\times} =  1$ implies that $\theta_\chi$ corresponds to a Maass form of weight 0 and Laplace eigenvalue $\ge 1/4$ if $K$ is real, and a holomorphic modular form of weight $\ge 1$ if $K$ is imaginary; moreover $\theta_\chi$ is a cusp form if and only if $\chi^2 \neq 1$.} on $\GL_2$ obtained as the theta lift of $\chi$. Let $f$ be a Maass newform of Laplace eigenvalue $\lambda$ on $\Gamma^D_0(p^n) \bs \H$. Then $$L(1/2, f \times \theta_\chi) \ll_{p,K, \chi, \lambda, \eps} (C(f \times \theta_\chi))^{5/24 +\eps}$$ where $L(s, f \times \theta_\chi)$ denotes the Rankin-Selberg $L$-function normalized to have functional equation $s \mapsto 1-s$, and $C(f \times \theta_\chi) \asymp_\chi p^{2n}$ denotes the (finite part of the) conductor of  $L(s, f \times \theta_\chi)$.}

\medskip

\begin{remark} The classical construction of the theta lift $\theta_\chi$ goes back to Hecke and Maass. This was generalized  in the representation-theoretic language by Shalika and Tanaka \cite{shalika-tanaka}; see also \cite[Sec. 13]{harris-kudla-1991} for a more modern treatment. For an explicit formula for $\theta_\chi$ under certain assumptions, see also page 61 of \cite{MR2061214} for the holomorphic case, and Appendix A.1 of \cite{humphries-khan} for the Maass case. The automorphic representation corresponding to $\theta_\chi$ is precisely the global \emph{automorphic induction} $\AI(\chi)$ of $\chi$ from $\A_K^\times$ to $\GL_2(\A)$. This is a special instance of the Langlands correspondence, as explained nicely in Gelbart's book \cite[7.B]{Gelbart1975}.
\end{remark}

\begin{remark} Thanks to the Jacquet-Langlands correspondence, we may equivalently take $f$ in Theorem B to be a newform on $\GL_2$ (of level equal to $\disc(D)p^n$). Theorem B may be viewed as a (depth-aspect) subconvexity result for $L(1/2, f \times g)$ where $g= \theta_\chi$ is fixed and $f$ varies. Subconvexity  for the Rankin-Selberg $L$-function on $\GL_2 \times \GL_2$ (with one of the $\GL_2$ forms fixed)  in the level-aspect was first approached by Kowalski-Michel-Vanderkam \cite{KMV02} and a complete solution was obtained by Harcos--Michel \cite{harcos-michel}. Uniform subconvexity in all aspects  was subsequently  proved in ground-breaking work of Michel-Venkatesh \cite{michel-2009}, who showed that $L(1/2, f \times g) \ll_{g, \eps} C(f\times g)^{1/4 - \delta}$ for general $f$ and $g$ and some $\delta>0$. There have also been recent works, notably by Han Wu, that make the (unspecified) exponent $\delta$ of Michel-Venkatesh explicit in various cases. We also remark that \begin{equation}\label{e:bc}L(1/2, f \times \theta_\chi) = L(1/2, f_K \times \chi)\end{equation} where $f_K$ denotes the \emph{base-change} of $f$ to $K$; so Theorem B may also be viewed as a special instance of subconvexity on $\GL_2(\A_K) \times \GL_1(\A_K)$ with the character on $\GL_1(\A_K)$ fixed. We further note that in the special case that $\chi = \mathbf{1}$ is trivial, the $L$-function factors as $L(1/2, f \times \theta_\mathbf{1})  = L(1/2, f \times \rho_K) L(1/2, f)$ where $\rho_K$ is the quadratic Dirichlet character associated to $K$.

As the above discussion makes clear, subconvexity in the setup of Theorem B is not new. However, the exponent $5/24$ (corresponding to $\delta = 1/24$) we obtain  appears to be the current strongest bound in this particular setting. As a point of comparison, the exponent that can be extracted in our setting from the general bound given in  Corollary 1.6 of \cite{Wu}, followed by an application of \eqref{e:bc}, corresponds to $\delta = \frac{1-2\theta}{32} < \frac{1}{24}$.
\end{remark}

The proof of Theorem B uses an explicit version of Waldspurger's famous formula \cite{Waldspurger1985} relating squares of toric periods and central $L$-values. We emphasize that the proof follows immediately upon substituting the bound from Theorem A into this formula, and does not need any additional ingredients.

\medskip

We now explain the main ideas behind Theorem A, and how they can be put into a general framework. The usual strategy\footnote{A notable exception being a recent preprint of Sawin \cite{sawin19} that treats the\emph{ function field analogue} of the level aspect sup-norm problem using a very different geometric method.} to prove a sup-norm bound in the level aspect is to use the amplification method. This involves choosing  a suitable global test function (a product of local test functions over all places) and then estimating the geometric side of the resulting pre-trace formula by counting the number of lattice points that lie in the support of the test function, as the level varies.  This strategy successfully works to beat the local bound in the squarefree level aspect, where one can choose the local test functions at the ramified primes to be the indicator function  (modulo the centre) of the local Hecke congruence subgroups. This strategy also works very well for families of automorphic forms corresponding to highly \emph{localized} vectors at the ramified places, such as the minimal vectors or the $p$-adic microlocal lifts; the corresponding sup-norm bounds in these cases were proved in \cite{saha-sup-minimal}.

Unfortunately, this strategy  on its own fails to beat the local bound in the depth aspect for newforms. The reason is that local newvectors are not sufficiently localized in the depth aspect, and consequently the support of the ``best" test function modulo the centre, as far as the depth aspect is concerned, is essentially the entire maximal compact subgroup. Therefore the support does not involve many congruence conditions, and congruences are essential
for achieving saving via counting. If we were to reduce the support of our ramified test functions further and thus force new congruences, the resulting saving via counting would be eclipsed by the resulting loss due to the fact that we will be averaging over more cusp forms.

 The key new contribution of this paper is that we focus not merely on the support of the test function, but instead quantify how fast the test function (which is essentially the matrix coefficient of the local newvector) decays within the support.  Roughly speaking, our method  divides up the geometric side of the (amplified) pre-trace formula into multiple pieces, corresponding to a filtration of the support of the local test function.  These pieces are estimated separately to obtain a general theorem that gives a sup-norm bound in the level aspect which is stronger than what can be obtained by existing methods. To illustrate our technique in the setting of Theorem A, for each level $N=p^n$ consider the filtration of compact subgroups $K^*(j) \subset K^*(j-1) \subset \ldots \subset K^*(1)$ of $\GL_2(\Z_p)$, where $j \asymp n/8$ and $K^*(i)$ is equal to the subgroup that looks like $\mat{\ast}{0}{0}{\ast}$ modulo $p^i$. The support of the test function at the prime $p$ is $K^*(1)$. We break up the geometric side of the pre-trace formula into $j$ pieces, with piece $i$ (for $1 \le i \le j$)  corresponding to the matrices whose local component at $p$ lies in $K^*(i)$ but not in $K^*(i+1)$ (where we let $K^*(j+1)$ denote the empty set). Now, we prove that these local matrix coefficients have a proper decay property, due to which the size of the test function at each matrix in piece $i$ is bounded\footnote{In fact, for the purpose of Theorem A, we only need the weaker bound that the size at piece $i$ is bounded by $p^{2i - \frac{n}{4}}$.} by $p^{\frac{i}{2} - \frac{n}{4}}$.  Therefore for  each piece, we get a saving  from two sources: (a) from the size of the test function, (b) from counting lattice points. The saving from source (a) is large when $i$ is small, which is precisely when the saving from source (b) is small. Conversely, when $i$ is large, the saving from source (a) is small and the saving from source (b) is large.  We emphasize that we are still using an amplified pre-trace formula, but with the extra ingredient described above, which leads to the bound in Theorem A.

The reader may have noticed that our exponent $5/24$ in Theorem A coincides with the exponent obtained by Iwaniec--Sarnak in \cite{iwan-sar-85}. In hindsight, our filtration strategy at a place $p$ is analogous to the argument used by Iwaniec and Sarnak in \cite[Lemma 1.1 - 1.3]{iwan-sar-85} for the test function at infinity in their classic work on the eigenvalue aspect of the sup-norm problem. Crucially, our bound  for the size of $p$-adic matrix coefficients, and that of Iwaniec--Sarnak for the archimedean matrix coefficient, both involve showing that
the coefficient decays away from a torus. The relation between the sup-norm problem and subconvexity is also similar. In particular, when $\chi$ is fixed in Theorem B, the  local bound agrees with the convexity bound for the central $L$-value, and the
automorphic period we consider is a sum over a fixed collection of points if $K$ is imaginary (resp., a sum of integrals over closed geodesics if $K$ is real), and subconvexity follows  from any improvement over the local bound with no
cancellation in the sum being required. This is analogous to what happens on applying the Iwaniec--Sarnak bound for Eisenstein series, where one obtains a  $t$-aspect subconvexity result for the Riemann zeta function (see Remark D of \cite{iwan-sar-85}).

On the other hand, the required bounds for the archimedean matrix coefficient used by Iwaniec and Sarnak (see Lemma 1.1 of \cite{iwan-sar-85}) follow in an elementary manner using integration by parts. However, our $p$-adic matrix coefficient is more subtle and so we need quantitative results on the decay  across a sequence of \emph{compact} subsets of matrix coefficients associated to local newvectors. Such results do not appear to be available in the literature; indeed, existing results on decay of matrix coefficients (e.g., see \cite{Oh02}) typically give the decay for the torus-component of the elements (in the sense of the Cartan decomposition) going to infinity, which are completely orthogonal to what we require. In Theorem \ref{p:localnewmatrixmain}(1), we provide a general quantitative statement about the decay of these matrix coefficients of the sort we need, which may be of independent interest. The proof of Theorem \ref{p:localnewmatrixmain} is carried out in Section \ref{s:local} (which can be read independently of the rest of the paper) and uses the stationary phase method in the $p$-adic context. A key role in the proof is played by an  useful formula\footnote{For some history of this type of formula, see Remark 2.20 of \cite{corbett-saha}.} for the Whittaker newvector in terms of a family of ${}_2F_1$
hypergeometric integrals, which allows us to use the $p$-adic stationary phase method.

The idea outlined above can be phrased in a more general context (without any need to restrict ourselves to newforms) to prove an improved sup-norm bound whenever suitable results on decay of local matrix coefficients along a filtration of compact subsets are available. We develop a suitable language for such a result in Sections \ref{s:localfam} and \ref{s:globalfamilies} leading to Theorem \ref{t:maingen}, which may be regarded as the ``master theorem" of this paper.  Theorem \ref{t:maingen} gives a strong sup-norm bound for any family of automorphic forms of powerful levels for which certain local hypotheses are satisfied. Thus it reduces the question of proving these bounds to checking these local hypotheses, and Theorem \ref{p:localnewmatrixmain}, described earlier, is essentially the statement that these local hypotheses are satisfied by the family of local newvectors of odd conductor and trivial central character. The proof of Theorem \ref{t:maingen}  is carried out in Section \ref{s:amplglobal} and uses as a main ingredient a lattice-point counting result proved in \cite{saha-sup-minimal}.

We end this introduction with a few remarks about possible extensions of this work. It should be possible to extend the argument to prove a non-trivial
hybrid bound (simultaneously in the depth and eigenvalue aspects) for the sup-norm, however we do not attempt to do so here. The method of this paper can be combined with  the Fourier/Whittaker expansion at various cusps in the adelic context (the necessary machinery for which is now available thanks to recent work of Assing \cite{assing18} building on earlier work of the second author \cite{sahasupwhittaker, saha-sup-level-hybrid}) to give a depth aspect sub-local bound in the case $D = M_2(\Q)$  (possibly with a different exponent than in Theorem A due to some differences in the counting argument).   Finally, this paper provides a general strategy of how one should go about improving the local bound in the level aspect in cases where the local vectors are not sufficiently localized. Essentially, the message is that one needs to combine a counting argument with a ``decay of matrix coefficients" argument to successfully attack this problem for a wide array of local and global families.

\section{Preliminaries}\label{s:global}
\subsection{Basic notations}\label{s:globalstatement}
The basic notations used in this paper are by and large the same as those in \cite{saha-sup-minimal}, but for convenience we recall them here.

 \subsubsection*{Generalities}Let $\f$ denote the finite places of $\Q$ (which we identify with the set of primes) and $\infty$ the archimedean place.  We let $\A$ denote the ring of adeles over $\Q$, and $\A_\f$ the ring of finite adeles.   Given an algebraic group $H$ defined over $\Q$, a place $v$ of $\Q$, a subset of places $U$ of $\Q$, and a positive integer $M$, we denote $H_v:= H(\Q_v)$, $H_U:=\prod_{v \in U} H_v$, $H_M :=\prod_{p|M} H_p$. Given an element $g$ in $H(\Q)$ (resp., in $H(\A)$), we will use $g_p$ to denote the image of $g$ in $H_p$ (resp., the $p$-component of $g$); more generally for any set of places $U$, we let $g_U$ denote the image of $g$ in $H_U$.

Given two integers $a$ and $b$, we use $a|b$ to denote that $a$ divides $b$, and we use $a|b^\infty$ to denote that $a|b^n$ for some positive integer $n$. For any real number $\alpha$, we let $\lfloor \alpha \rfloor$ denote the greatest integer less than or equal to $\alpha$ and we let $\lceil \alpha \rceil$ denote the smallest integer greater than or equal to $\alpha$. For any integer $A=\prod_{p\in \f} p^{a_p}$, we write \begin{equation}\label{A1} A_1 = \prod_{p\in \f}  p^{\lceil \frac{a_p}2 \rceil} \end{equation} In other words, $A_1$ is the smallest integer such that $A$ divides $A_1^2$.

All representations of (topological) groups are assumed to be continuous and over the field of complex numbers.

\subsubsection*{Quaternions, orders, and groups} Throughout this paper, we fix an indefinite quaternion division algebra $D$ over $\Q$. We fix once and for all a maximal order $\O^\m$ of $D$. All constants in the bounds in this paper will be allowed to depend on $D$ without explicit mention. We let $d$ denote the reduced discriminant of $D$, i.e., the product of all primes such that $D_p$ is a division algebra. We let $\nr$ be the reduced norm on $D^\times$.

We denote $G=D^\times$ and $G' =PD^\times = D^\times/Z$ where $Z$ denotes the center of $D^\times$.
   For each prime $p$, let $K_p=(\O^\m_p)^\times$ and let $K_p'$ denote the image of $K_p$ in $G'_p$. Given an order $\O$ of $D$, we define a compact open subgroup of $G(\A_\f)$ by $$K_\O =\prod_{p \in \f} \O_p^\times.$$

 For each place $v$ that is not among the primes dividing $d$, fix once and for all an isomorphism $\iota_v: D_v  \xrightarrow {\cong} M(2,\Q_v)$. We assume that these isomorphisms are chosen such that for each finite prime $p\nmid d$, we have $\iota_p(\O_p) = M(2,\Z_p)$.  By abuse of notation, we also use $\iota_v$ to denote the composition map $D(\Q) \rightarrow D_v \rightarrow M(2,\Q_v)$.

For any lattice $\L\subseteq \O^\m$ of $D$, we get a local lattice $\L_p$ of $D_p$ by localizing at each prime $p$. These collection of lattices satisfy \begin{equation}\label{e:localization} \L = \{g \in D: g_p \in \L_p \text{ for all primes }p\}.\end{equation}  Conversely, if we are given a collection of local lattices $\{\L_p\}_{p \in \f}$, such that $\L_p \subseteq \O^\m_p$ for all $p$ and $\L_p = \O^\m_p$ for all but finitely many $p$, then there exists a unique lattice $\L\subseteq \O^\m$ of $D$ defined via \eqref{e:localization} and whose localizations at primes $p$ are precisely the $\L_p$.  We will refer to $\L$ as the global lattice corresponding to the collection of local lattices $\{\L_p\}_{p \in \f}$. More generally, given a finite subset $S \subseteq \f$, and  a collection of local lattices $\{\L_p\}_{p \in S}$, we can construct the (unique) lattice whose localization at a prime $p$ equals $\L_p$ if $p\in S$ and equals $\O_p^\m$ if $p \notin S$; we will refer to this lattice as the global lattice corresponding to $\{\L_p\}_{p \in S}$.

Let $\L$ be a lattice in $D$ such that $\L\subseteq \O^\m$. We say that $\L$ is \emph{tidy} in $\O^\m$ if $\L$ contains 1, and $M_3^2$ divides $N=[\O^\m: \L]$ where $(M_1, M_2, M_3)$ are the unique triple of positive integers such that $M_1|M_2|M_3$ and $ \O^\m /\L \simeq (\Z / M_1\Z) \times (\Z / M_2\Z) \times (\Z / M_3\Z).$ Note that since $N=M_1M_2M_3$, $M_3^2$ divides $N$ if and only if $N$ divides $(M_1M_2)^2$ if and only if $M_3$ divides $M_1M_2$. Let $\L_p$ be a lattice of $D_p$ such that $\L_p\subseteq \O_p^\m$. We say that $\L_p$ is tidy in $\O_p^\m$ if $1 \in \L_p$ and $m_3 \le m_1+m_2$, where $(m_1, m_2, m_3)$ are the unique triple of non-negative integers such that $m_1\le m_2 \le m_3$ and $\O^\m_p/\L_p \simeq (\Z_p / p^{m_1}\Z_p) \times (\Z_p / p^{m_2}\Z_p)  \times (\Z_p / p^{m_3}\Z_p).$ It is clear that a global lattice $\L$ is tidy in $\O^\m$ if and only if all the corresponding local lattices $\L_p$ are tidy in $\O^\m_p$.

For each $g \in G(\A_\f)$, and a lattice $\L$ of $D$, we let ${}^g\L$ denote the lattice whose localization at each prime $p$ equals $g_p \L_p g_p^{-1}$. Note that if $g \in K_{\O^\m}$, and $\L$ is tidy in $\O^\m$, then  ${}^g\L$ is also tidy in $\O^\m$.

\subsubsection*{Haar measures}We fix the Haar measure on each group $G_p$ such that $\vol(K_p)=1$. We fix a Haar measure on $\Q_p^\times$ such that $\vol(\Z_p^\times)=1$. This gives us resulting Haar measures on each group $G'_p$ such that $\vol(K'_p)=1$.  Fix any Haar measure on $G_\infty$, and take the Haar measure on $\R^\times$ to be equal to $\frac{dx}{|x|}$ where $dx$ is the Lebesgue measure. This gives us a Haar measure on $G'_\infty$. Take the measures on $G(\A)$ and $G'(\A)$ to be given by the product measure.  

 For each continuous function  $\phi$ on the space $G(\A)$, we let $R(g)$ denote the right-regular action, given by $(R(g)\phi)(h) = \phi(hg)$. If a  continuous function  $\phi$ on $G(\A)$ satisfies that $|\phi|$ is left $Z(\A)G(\Q)$ invariant, define
\begin{equation}\label{e:l2norm}\|\phi\|_2 = \left(\int_{G'(\Q)\bs G'(\A)} |\phi(g)|^2 dg\right)^{1/2}.\end{equation}
Note above that $G'(\Q)\bs G'(\A)$ is compact, so convergence of the integral is not an issue.

 \subsubsection*{Asymptotic notation}
 We use the notation
$A \ll_{x,..,y} B$
to signify that there exists
a quantity $C$ depending only on $x,..,y$ (and possibly on any objects fixed throughout the paper) so that
$|A| \leq C |B|$. We use $A \asymp_{x,..y} B$ to mean that $A \ll_{x,..y} B$ and $B \ll_{x,..y} A$.
 The symbol $\epsilon$ will denote a small positive quantity whose value may change from line to line; a statement such as $A \ll_{\eps, x,..} B$ should be read as ``For all small $\eps>0$, there is a quantity $C$ that depends only on $\eps, x,..,$ and on any objects fixed throughout the paper,  such that $|A| \leq C |B|$." An assertion such as $A \ll_{x,..y} D^{O(1)} B$ means that there is a constant $C$ such that $|A| \ll_{x,..y} |D|^{C} |B|$.  
\subsection{A counting result}Let $u(z_1, z_2)= \frac{|z_1 - z_2|^2 }{4 \Im(z_1)\Im(z_2)}$, which is a function of the usual hyperbolic distance on $\H$. For the convenience of the reader, we recall a counting result from \cite{saha-sup-minimal} that will be used later.

\begin{proposition}\label{t:counting} For a compact subset $\J$  of $\H$ and a tidy lattice  $\L \subseteq \O^\m$ with $[\O^\m: \L]=N$, the following bounds hold for all $z \in \J$,

\begin{equation}\label{e:bd1}\sum_{1\le m \le L } |\{\alpha \in \L: \nr(\alpha)=m, u(z, \iota_\infty(\alpha) z) \le \delta\}| \ll_{\epsilon, \delta, \J} L^\epsilon N^{\epsilon} \left(  L +\frac{L^2}{N} \right).\end{equation}
 \begin{equation}\label{e:bd2}\sum_{1\le m \le L } |\{\alpha \in \L: \nr(\alpha)=m^2, u(z, \iota_\infty(\alpha) z) \le \delta\}|  \ll_{\epsilon, \delta,\J} L^\epsilon  N^{\epsilon} \left(L  +\frac{L^3}{N}   \right).\end{equation}

\end{proposition}
\begin{proof}This is an immediate corollary of \cite[Proposition 2.8 and Remark 2.11]{saha-sup-minimal}. Note that Proposition 2.8 of \cite{saha-sup-minimal} had the additional assumption  $1 \le L \le N^{O(1)}$. However, as is clear from the proof of that Proposition, the assumption was used there for the sole purpose of replacing any $L^\epsilon$ factors by $N^\epsilon$. Here, we have removed that assumption and instead included additional  factors of $L^\epsilon$ on the right sides of each of \eqref{e:bd1}, \eqref{e:bd2}.
\end{proof}
\begin{remark}The above result is the main reason why we introduced the concept of ``tidy". For non-tidy lattices, the counting result gets more complicated as demonstrated in Proposition 2.8 of \cite{saha-sup-minimal}.
\end{remark}

\section{Local families}\label{s:localfam}For each prime $p \in \f$, we let $\Pi(G_p)$ denote the set of isomorphism classes of
representations $\pi$ of $G_p$ that are irreducible, admissible, unitary, and if $p\nmid d$, also  infinite-dimensional.
Let $$\Aa_p = \{(\C v, \pi): \pi \in \Pi(G_p), 0 \neq v \in V_\pi\}.$$ \begin{definition}A \emph{local family} (over $G_p$) is a subset of $\Aa_p$.
\end{definition}
We will typically use $\F_p$ to denote a local family over $G_p$ and sometimes write the elements of $\F_p$ as $\F_p = \{(\C v_{i,p} , \pi_{i,p})_{i \in S_p}\}$ where $S_p$ denotes an indexing set.


\begin{definition}For each $p \in \f$, we let $\F_p^{\ur}$  denote the local family consisting of all the pairs $(\C v, \pi)$ such that $\pi \in \Aa_p$ has the unique $K_p$-fixed line $\C v$.
\end{definition}

For each $p \nmid d$, $\pi \in \Pi(G_p)$,  let $a(\pi) \in \Z_{\ge 0}$ denote the exponent in the conductor of $\pi$. We write $a_1(\pi) = \lceil \frac{a(\pi)}{2}\rceil.$

\begin{definition}A \emph{nice} local family over $G_p$ is a subset $\F_p$ of $\Aa_p$ with the following properties: \begin{enumerate}

\item If $p|d$ then $\F_p = \F_p^\ur$.

\item If $p\nmid d$, then $$\F_p \cap  \{(\C v, \pi): \pi \in \Pi(G_p), a(\pi)=0\} =  \F_p^\ur.$$

\end{enumerate}

\end{definition}

\begin{definition} A \emph{nice collection of local families} (or simply, a \emph{nice collection})   is a tuple of the form  $\F = (\F_p)_{p\in \f}$ such that  for each prime $p \in \f$, $\F_p$ is a nice local family over $G_p$.

\end{definition}

\begin{remark}   Note that a nice  local family does not have any ``old-vectors" originating from spherical (i.e., $K_p$-fixed) vectors. Furthermore, nice collections have no complications at the places dividing $d$.  We will restrict to nice families/collections for technical convenience and to get a cleaner statement of our main global theorem later on.
\end{remark}

The following definition quantifies the decay of matrix coefficient along a filtration of compact subsets,
needed for our main theorem.

\begin{definition}\label{d:uniformlycontrolled} Let $\eta_1$, $\eta_2$, $\delta$ be non-negative real numbers such that $\eta_1 \le \eta_2$. Let $\F=(\F_p)_{p\in \f}$  be a nice collection, and for each $p \in \f$, write $\F_p = \{(\C v_{i,p}, \pi_{i,p})_{i\in S_p}\}$ where $S_p$ is any indexing set for $\F_p$. We say that $\F$ is \emph{controlled} by $(\eta_1, \delta; \eta_2)$ if there exists $c\ge0$ ($c$ depending only on $\F, \eta_1, \eta_2$), and furthermore, for each $p\nmid d$ and $i \in S_p$ such that $a(\pi_{i,p})>0$, there exists an element $g_{i,p} \in G_p$,  so that using the shorthand \[v'_{i,p}:=\pi_{i,p}(g_{i,p})v_{i,p}, \quad \Phi'_{i,p}(g) = \frac{\langle \pi_{i,p}(g) v'_{i,p}, v'_{i,p}\rangle}{\langle v'_{i,p}, v'_{i,p} \rangle},\] the conditions (1), (2) below hold for each $p\in \f$, $p\nmid d$, $i \in S_p$ for which $a(\pi_{i,p})>0$,
    \begin{enumerate}
\item There exists a tidy order  $\O_{i,p} \subseteq \O_p^\m$, such that
    \begin{enumerate}
    \item $[\O_p^\m:\O_{i,p}] \ll  p^{\eta_1 a_1(\pi_{i,p}) + c(\eta_2 - \eta_1)}$,
    \item The $\pi_{i,p}$-action of $\O_{i,p}^\times$  on $v_{i,p}'$  generates an irreducible representation of dimension $\ll  p^{\delta a_1(\pi_{i,p})}$.
    \end{enumerate}

\item For each $\eta$ satisfying $\eta_1 < \eta \le \eta_2$, there exists a tidy lattice $\L^\eta_{i,p} \subseteq \O_{i,p}$, such that
    \begin{enumerate}

    \item $\L^{\eta'}_{i,p} \subseteq \L^{\eta}_{i,p}$ for all $\eta_1<\eta \le \eta' \le \eta_2$,

    \item $p^{\eta a_1(\pi_{i,p})-c} \ll [\O_p^\m:\L^\eta_{i,p}] \ll  p^{\eta a_1(\pi_{i,p})+c}$,

     \item If $g \in  \O_{i,p}^\times$, $g \notin \L^\eta_{i,p}$, we have $|\Phi'_{i,p}(g)| \ll p^{c+ (\eta-\eta_2) a_1(\pi_{i,p})}.$
    \end{enumerate}

\end{enumerate}

\end{definition}

\begin{remark}Suppose we have a collection $\F$ which is controlled by $(\eta_1, \delta; \eta_2)$. Then it is trivially true that  $\F$ is controlled by $(\eta_1, \delta; \eta_2')$ for any $\eta_1 \le \eta_2' \le \eta_2$. Therefore, whenever we assert that $\F$ is controlled by some $(\eta_1, \delta; \eta_2)$ we will try and ensure that we choose $\eta_2$ as large as possible (for those particular values of $\eta_1$ and $\delta$).
\end{remark}

\begin{remark}Suppose that $\F$ is controlled by $(\eta_1, \delta; \eta_2)$. Let us explore the possible range of values that $\eta_1, \eta_2, \delta$ can take. We assume for the purpose of this remark that for each prime $p$ either $\F_p = \F_p^\ur$ or the set $\{a(\pi_{i,p}): i \in S_p\}$ is unbounded.

We first focus on the implications of condition (1). Let $i\in S_p$ with $a(\pi_{i,p})>0$. Then, condition (1) implies that \begin{equation}\label{e:integral}\int_{\O_{i,p}^\times}|\Phi'_{i,p}(g)|^2 dg\gg p^{(-\delta - \eta_1) a_1(\pi_{i,p}) - c(\eta_2 - \eta_1)}.\end{equation}Now, it can be shown (by formal degree considerations) that for $\pi_{i,p}$ discrete series, the left hand side above is $\ll p^{-a_1(\pi_{i,p})}$. In fact, an explicit computation (performed in \cite{saha-sup-level-hybrid}) shows that the same holds for principal series. Therefore (by letting $i \rightarrow \infty$) we obtain the inequality \begin{equation}\label{e:inequality1} \eta_1 + \delta \ge 1.\end{equation} This inequality is sharp in the sense that there exist several collections $\F$ that satisfy condition (1) for some $\eta_1$, $\delta$ with $\eta_1 + \delta = 1$. Indeed, for many natural collections (including those that correspond locally to newvectors of trivial character, minimal vectors, and $p$-adic microlocal lifts)  one can choose the order $\O_{i,p} = \O^\m_p$  to ensure that the condition (1) of Definition \ref{d:uniformlycontrolled} holds with $\eta_1=0$, $\delta = 1$; see Proposition 2.13 of \cite{saha-sup-level-hybrid},   Section 1.4 and Remark 3.2 of \cite{saha-sup-minimal}, and Corollary A.3 of \cite{HN-minimal}.

Next we explore what is the possible range of values that $\eta_2$ can take given $\eta_1$ and $\delta$. Combining \eqref{e:integral} with condition (2) of Definition \ref{d:uniformlycontrolled} and the triangle inequality, a simple computation leads to \begin{equation}\label{e:inequality}\eta_2  \le \eta_1 + \delta.\end{equation} On the other hand suppose we have a collection $\F$ satisfying condition (1) of Definition \ref{d:uniformlycontrolled} for some $\eta_1$, $\delta$. Then, it is trivially true that $\F$ is controlled by $(\eta_1, \delta; \eta_1)$.

So, to \emph{summarize}, if a collection $\F$ satisfies condition (1) of Definition \ref{d:uniformlycontrolled} for some $\eta_1$, $\delta$, then \eqref{e:inequality1} holds, and if we then want to find some $\eta_2$ such that $\F$ is controlled by $(\eta_1, \delta; \eta_2)$, then any such $\eta_2$ must lie in the range $[\eta_1, \eta_1 + \delta]$. In this range, $\eta_2=\eta_1$ always works.
An interesting question, and one which we do not know the answer to, is the following: \emph{Suppose a collection satisfies condition (1) for some $\eta_1$, $\delta$ with $\delta>0$. Can we always find some $\eta_2>\eta_1$ such that $\F$ is controlled by $(\eta_1, \delta; \eta_2)$}?
\end{remark}

\begin{remark} In relation to the last remark, the main result of \cite{saha-sup-minimal} tells us that whenever a collection satisfies condition (1) of Definition \ref{d:uniformlycontrolled}  with $\eta_1=\eta_2$ and $\frac{\eta_1}3 + \frac{\delta}2 <\frac{1}2$, then we can break the \emph{local bound} for the sup-norms of the corresponding global automorphic forms. Unfortunately it is not always true that naturally occurring collections have this property.

The crucial \emph{new} ingredient in this paper is represented by the condition $(2)$, which posits a linear decay result for the matrix coefficient associated to a suitable translate of $v_{i,p}$. Whenever we can prove a quantitative decay of local matrix coefficients so that $\F$ is controlled by $(\eta_1, \delta; \eta_2)$ for some $\eta_2>\eta_1$, it will allow us (in our main global theorem, Theorem \ref{t:maingen} below) to improve upon the sup-norm estimate obtained from condition (1) alone.
\end{remark}

\begin{remark}The assumption  that the relevant lattices/orders in Definition \ref{d:uniformlycontrolled} are tidy is in order to get a cleaner statement of Theorem \ref{t:maingen} later on. However, this is not essential for our method and one could in principle omit from Definition \ref{d:uniformlycontrolled} the condition that the lattices are tidy. However, in that case, Proposition \ref{t:counting} would need to be modified and Theorem \ref{t:maingen} below would get  more complicated.
\end{remark}

\begin{remark}One could refine Definition \ref{d:uniformlycontrolled} by including the constant $c$ among the ``controlling" parameters, or by replacing $c$ with a function of $i$ and $p$. Any such hybrid definition can be used to make a refinement of Theorem \ref{t:maingen} below without much additional work. We avoid doing this in this paper in the interest of simplicity, and because our main focus is in the depth aspect.
\end{remark}

\begin{example}   For each $p\nmid 2d$, define the local family $\F_p^{\rm min, *}$ to be the union of $\F_p^\ur$ and all pairs $(\C v, \pi)$ such that  $\pi$ is a twist-minimal supercuspidal representation of $G_p$ satisfying $a(\pi) \not \equiv 2 \pmod{4}$ and $v$ is a minimal vector in $\pi$ in the sense of \cite{HN-minimal}. For $p|2d$, define $\F_p^{\rm min, *} = \F_p^\ur$. Let $\F^{\rm min, *}$ be the corresponding nice collection. Then by the results of \cite{HN-minimal}, $\F^{\rm min, *}$ is  controlled by $(1,0;1)$.
 Furthermore, it follows from Remark 3.2 of \cite{saha-sup-minimal} that $\F^{\rm min, *}$ is controlled by $(\gamma, 1-\gamma;1)$ for all $0\le \gamma \le 1$. So, this is an example where equality is attained in both \eqref{e:inequality1} and \eqref{e:inequality}.

\end{example}

\begin{definition}Let $p \nmid d$ be a prime. Define the nice local family $\F_p^{\rm new, *}$ to consist of all pairs $(\C v, \pi)$ with $\pi$ varying over the representations in $\Pi(G_p)$ with \emph{unramified} central character, and $\C v$ equal to the (unique) line generated by the local \emph{newvector}.
\end{definition}

The following result will follow from our work in Section \ref{s:local} of this paper.

\begin{proposition}\label{t:mainlocal}Let $\mathcal{G} = \{\G_p\}$ be the nice collection given by \begin{enumerate}

\item $\G_p = \F_p^{\rm new, *}$ if $p\nmid 2d$,

\item $\G_p = \F_p^\ur$ if $p|2d$.

\end{enumerate}
 Then $\G$ is controlled by $(0,1; \frac12)$.
\end{proposition}

\begin{remark} Roughly speaking, Proposition \ref{t:mainlocal} asserts (among other things) that for each fixed odd prime $p$ and each local representation $\pi_p$ of $\GL_2(\Q_p)$ with $a_1(\pi_p)=n_1$, there is a certain translate $v'$ of the newform whose associated matrix coefficient $\Phi'(g)$ is bounded by $p^{-n_1/2} \ [\O_p^\m:\L^\eta_p]$ at matrices $g \notin \L^\eta_p$, where $\{\L_p^{\eta}\}_{0 \le \eta \le \frac12} $ is  a suitable filtration of lattices in $\O_p^\m$ such that $[\O_p^\m:\L^\eta_p] \asymp p^{\eta n_1 + O(1)}$.

However, what we will end up proving in Section  \ref{s:local} is the stronger statement that the matrix coefficient $\Phi'(g)$ is bounded by $p^{-n_1/2} \ [\O_p^\m:\L^\eta_p]^{1/4}$ at such matrices.

Unfortunately, this stronger bound does not help in improving the exponent $5/24$ in Theorem A. This is essentially because both the above bounds coincide when  $[\O_p^\m:\L^\eta_p] \asymp 1$.
\end{remark}

\begin{remark}\label{rem:localsmallexclude}Let $k_0$ be some \emph{fixed} non-negative integer. For each prime $p$  not dividing $2d$ consider the subset of $\F_p^{\rm new, *}$ consisting of the pairs $(\C v_i, \pi_i) \in \F_p^{\rm new, *}$ where $a(\pi_i) \le k_0$. Then letting $g_{i,p} = \iota_p^{-1}\mat{p^{a_1(\pi_i)}}{}{}{1}$, and $\L^\eta_{i,p} = \O^\m_{i,p}$, we see that the conditions in Definition \ref{d:uniformlycontrolled} hold (trivially) for $\eta_1 =0$, $\delta=1$, $\eta_2 = \frac12$, with the constant $c$  equal to $\frac{k_0}2$. So, in order to prove Proposition \ref{t:mainlocal}, it suffices to restrict our attention only to representations $\pi_i$ with $a(\pi_i)>k_0$. We will use this with $k_0 = 2$ in Section \ref{s:local} when we prove the above Proposition.

Furthermore, for the proof of Proposition \ref{t:mainlocal}, it suffices to restrict ourselves only to the pairs $(\C v, \pi) \in \F_p^{\rm new, *}$ where $\pi_i$ has \emph{trivial} central character. This is because any unitary representation of $\GL_2(\Q_p)$ with unramified central character can be twisted by $| \det(g) |_p^s$ for some suitable $s \in i\R$ to make it have trivial central character; the twisting action in this case takes newforms to newforms, and the matrix coefficients etc., remain the same.
\end{remark}

\begin{remark}We suspect that Proposition \ref{t:mainlocal} continues to hold for the larger collection where we allow a) $p=2$, and b) replace the condition of unramified central character with more general central characters. However, for simplicity, we restrict ourselves to this case.
\end{remark}

\section{The main global result}Throughout this section, we will use the notations defined in Section \ref{s:globalstatement} and Section \ref{s:localfam}.
\subsection{Global families}\label{s:globalfamilies} We let $\Pi(G)$ denote the set of irreducible, unitary, cuspidal automorphic
representations of $G(\A)$. For any $\pi = \otimes_v \pi_v$ in $\Pi(G)$, we let $C(\pi)=\prod_{p\nmid d} p^{a(\pi_p)}$ denote the conductor\footnote{The conductor of $\pi$ equals $C(\pi)\prod_{p|d}p^{a(\pi_p)}$; thus  $C(\pi)$ denotes the ``away from $d$" part of the conductor of $\pi$. For $p|d$, $a(\pi_p)$ can be defined via the local Jacquet-Langlands correspondence; in particular, this gives $a(\pi_p)=1$ if $p|d$ and $\pi_p$ is one-dimensional.}   of $\otimes_{p\nmid d}\pi_p$, and we identify $V_\pi$ with a (unique) subspace of functions on $G(\A)$ so that $\pi(g)$ coincides with the right-regular representation $R(g)$ on that subspace. We define the integer $C_1(\pi)$ as in \eqref{A1}; i.e., $C_1(\pi)$ is  the smallest integer such that $C(\pi)$ divides $C_1(\pi)^2$. For any $\pi \in \Pi(G)$, define $$\ram(\pi) = \{p \in \f: p|C(\pi)\} = \{p \in \f: p \nmid d, \  \pi_p \text{ has no } K_p\text{-fixed line}\},$$ $$C'(\pi) =\prod_{p \in \ram(\pi)} p.$$

  We denote $$\Aa(G) = \{(\C \phi, \pi): \pi \in \Pi(G), 0 \neq \phi \in V_\pi\}.$$ If $\phi$ is a function such that $(\C \phi, \pi) \in \Aa(G)$, then  $|\phi|$ is left $Z(\A)G(\Q)$ invariant and hence we  define $\|\phi\|_2$ as in \eqref{e:l2norm}. For any such  $\phi$, we say that $\phi$ is \emph{factorizable} if $\phi$ corresponds to a pure tensor under the isomorphism\footnote{Such an isomorphism is unique up to scalar multiples, and we fix a choice of isomorphism once and for all.} $\pi \simeq \otimes_{v} \pi_v$, in which case we write $\phi = \otimes_v \phi_v$ with $\phi_v$ a vector in $\pi_v$.

\begin{definition}\label{d:arch} For $(\C\phi, \pi) \in \Aa(G)$, and $T>0$, we say that the archimedean parameters of $(\C\phi, \pi)$ are bounded by $T$ if the following two conditions hold: a) the analytic conductor $\q_\infty(\pi_\infty)$ (see \cite[p. 95]{MR2061214} for the definition) of $\pi_\infty$ satisfies $\q_\infty(\pi_\infty) \le T$, and b) the weight-vector decomposition of $\phi$ under the action of $\iota_\infty^{-1}(\SO(2))$ involves only weights $k$ such that $|k|\le T$.
\end{definition}

\begin{remark}\label{rem:kinftyfinite}Let $\phi$ be a cuspidal automorphic form on $G(\A)$ that generates some representation $\pi \in \Pi(G)$. Then it is easy to see that  $(\C\phi, \pi)$ has its archimedean parameters bounded by \emph{some} $T$ (since the usual definition of an automorphic form implies that $\phi$ is $K_\infty$-finite).
\end{remark}

\begin{definition}Given a nice collection $\F=(\F_p)_{p \in \f}$ of local families, we define the corresponding global family of automorphic forms $\Aa(G; \F)$ as follows: $$\Aa(G; \F) = \{(\C\phi, \pi) \in \Aa(G): \phi =\otimes_v \phi_v \text{ is factorizable, } (\C\phi_p, \pi_p) \in \F_p \text{ for all } p \in \f\}.$$ \end{definition}

 \begin{definition}For each $T>0$ we let $\Aa(G; \F, T) \subset \Aa(G; \F)$ consist of all the $(\C\phi, \pi)$ in $\Aa(G; \F)$  whose archimedean parameters are bounded by $T$.
 \end{definition}

 \begin{remark}Suppose that $\F$ is a nice collection and $(\C\phi, \pi) \in \Aa(G; \F,T)$. Then our definition of a nice collection implies that $$\{p \in \f: \phi_p \text{ is not } K_p\text{-fixed}\} = \ram(\pi).$$
\end{remark}

 \subsection{Statement of the main theorem}

 We can now state the master theorem of this paper.

\begin{theorem}\label{t:maingen}Let $\eta_1$, $\eta_2$, $\delta$, be non-negative real numbers such that $\eta_1 \le \eta_2$. Let $\F=(\F_p)_{p \in \f}$ be a nice collection that is controlled by $(\eta_1, \delta; \eta_2)$.  Then there is a non-negative constant $x$ depending only on $\F$, $\eta_1$, $\eta_2$ (we can take $x=0$ if $\eta_1 = \eta_2$) such that for all  $(\C\phi, \pi) \in \Aa(G; \F, T)$  we have  $$\sup_{g \in G(\A)} |\phi(g) | \ll_{T,  \eps} C'(\pi)^x C_1(\pi)^{\frac{\delta}{2} + \frac{\eta_1}{2} - \frac{\eta_2}6 +\eps} \|\phi \|_2.$$
\end{theorem}
The above Theorem can be viewed as a generalization of Theorem 1 of \cite{saha-sup-minimal}, which dealt with the special case\footnote{Note however that in \cite{saha-sup-minimal} we  did not assume that the relevant orders are tidy.} $\eta_1 = \eta_2$; in this special case, condition (2) of Definition \ref{d:uniformlycontrolled} is vacuous and does not play any part.
\begin{remark}In previous sup-norm papers such as \cite{HNS, sahasupwhittaker}, we often restricted to automorphic forms which corresponded classically to Hecke eigenforms that are either Maass cusp forms of weight $0$ or holomorphic cusp forms of weight $k$. Definition \ref{d:arch} above (see also Remark \ref{rem:kinftyfinite}) allows us to state Theorem \ref{t:maingen} for much more general automorphic forms.
\end{remark}

\begin{remark}As mentioned earlier, for many nice collections, the condition (1) of Definition \ref{d:uniformlycontrolled} holds with $\eta_1=0$, $\delta = 1$. This gives us the ``local bound"  \begin{equation}\label{e:localbd}\sup_{g \in G(\A)} |\phi(g) | \ll_{T, \eps} C_1(\pi)^{1/2 +\eps} \|\phi \|_2\end{equation} for any $\phi$ belonging to the corresponding global family of automorphic forms. Theorem \ref{t:maingen} gives us a pathway to go beyond \eqref{e:localbd} in this case whenever we can prove the existence of some $\eta_2 >  0$ for which condition (2) of Definition \ref{d:uniformlycontrolled} holds.

That this can indeed be done (with $\eta_2 = \frac12$) for the collection corresponding to global \emph{newforms} of odd conductor and trivial character is precisely the content of Proposition \ref{t:mainlocal}. This leads to the following corollary.
\end{remark}

\begin{corollary}\label{c:mainnewforms} Let $\mathcal{G}$ be as in Proposition \ref{t:mainlocal}. Let $C$ be a positive integer such that $(C, 2d)=1$, and let $C'$ be the product of all the primes dividing $C$. Let $(\C\phi, \pi) \in \Aa(G; \mathcal{G})$ and assume that \begin{enumerate}
\item $C(\pi)=C$,
\item $\phi_\infty$ is a vector of weight $k$ in $\pi_\infty$.
\end{enumerate}
Then we have
\begin{equation}\label{e:cormain}
\sup_{g \in G(\A)} |\phi(g) | \ll_{k, \pi_\infty,  \eps} (C')^{O(1)}C^{\frac{5}{24} + \eps}\|\phi\|_2. \end{equation}
\end{corollary}
\begin{proof} Clearly $\phi$ belongs to $\Aa(G; \mathcal{G}, T)$ for $\mathcal{G}$ as given by Proposition \ref{t:mainlocal} and $T$ depending only on $\pi_\infty$ and $k$. By Proposition \ref{t:mainlocal}, $\mathcal{G}$ is controlled by $(0, 1; \frac12)$. Now the result follows from Theorem \ref{t:maingen}.
\end{proof}

\begin{remark} It will be clear from the results of Section \ref{s:local}  that the exponent of $C'$ implicit in Corollary \ref{c:mainnewforms} is effective and can be written down explicitly.
\end{remark}

\subsection{The proof of Theorem \ref{t:maingen}}\label{s:amplglobal}
In this subsection, we complete the proof of Theorem \ref{t:maingen}. The case $\eta_1 = \eta_2$ is a direct corollary of Theorem 1 of \cite{saha-sup-minimal}. So throughout this proof we will assume that $\eta_2 > \eta_1$.

Let $\F$ be a nice collection that is controlled by $(\eta_1, \delta; \eta_2)$.
 Let $(\C\phi, \pi) \in \Aa(G; \F, T)$ be such that $\langle \phi, \phi \rangle =1$. Furthermore, we assume without loss of generality that $\phi$ is a weight vector, i.e., there exists some integer $k$ such that $|k| \le T$ and for all $g \in G(\A)$, \begin{equation}\label{weighteq}\phi\left( g \left( \iota_\infty^{-1}\mat{\cos(\theta)}{\sin(\theta)}{-\sin(\theta)}{\cos(\theta)}\right) \right) = e^{ik \theta}\phi(g).\end{equation}

 Henceforth we drop the index $i$ (since we are dealing with a particular $\phi$). Thus,  for each prime $p \in \ram(\pi)$,  the vector $v_{i,p}$ occurring in Definition \ref{d:uniformlycontrolled} is the vector $\phi_p$ in $\pi_p$ in the current setup. We let $\phi'_p$  be the local translate of $\phi_p$  that corresponds to $v'_{i,p}$ from  Definition \ref{d:uniformlycontrolled} for $p \in \ram(\pi)$; we define $\phi'_p = \phi_p$ for $p \notin \ram(\pi)$. We let $\phi'$ be the automorphic form on $G(\A)$ under the fixed isomorphism $\pi = \otimes_v \pi_v$. Then, the automorphic form $\phi'$ is just a translate of $\phi$ by a certain element of $G(\A_\f)$. Therefore, $\|\phi'\|_2 = \|\phi\|_2 =1$ and $\sup_{g \in G(\A)}|\phi'(g)| = \sup_{g \in G(\A)}|\phi(g)|$. Henceforth we will just work with $\phi'$.

 Given some $p \in \ram(\pi)$ and some $\eta_p$ such that $\eta_1 < \eta_p \le \eta_2$, let $\O_p$ and $\L^{\eta_p}_p$ satisfy the relevant conditions of Definition \ref{d:uniformlycontrolled}. For notational convenience, we henceforth denote $\L_p^{\eta_1} = \O_p$ for each $p \in \ram(\pi)$, so that $\L^{\eta_p}_p$ makes sense for the entire range $\eta_1 \le \eta_p \le \eta_2$

  Let $\O$ be the global order in $D$ corresponding to the collection of local orders $\{\O_p\}_{p \in \ram(\pi)}$. For any $\ram(\pi)-$\emph{tuple }$H = (\eta_p)_{p \in \ram(\pi)}$ with each $\eta_p$ chosen such that $\eta_1 \le \eta_p \le \eta_2$, let $\L^H$ be the global lattice such that $(\L^H)_{p} = \L^{\eta_p}_p$ if $p \in \ram(\pi)$ and $(\L^H)_{p} = \O^\m_p$ if $p \notin \ram(\pi)$.  Note that $\L^H\subseteq \O \subseteq \O^\m$ and the lattice ${}^g\L^H$ is tidy in $\O^\m$ for all choices of $H$ and all $g \in K_{\O^\m}$. We put $N=[\O^\m: \O]$, $N^H=[\O^\m: \L^H]$ and note that $N^H=N$ if $\eta_p=\eta_1$ for all $p \in \ram(\pi)$. By our assumptions (see Definition \ref{d:uniformlycontrolled}) we have \begin{equation}\label{e:levelbds}N^H \ll {C'(\pi)^{O(1)}} C_1(\pi)^{\eta_p}, \quad  N/N^H \ll {C'(\pi)^{O(1)}}C_1(\pi)^{\eta_1-\eta_p}. \end{equation}

Let $\J$ be a fixed (compact) fundamental domain for the action of $$\Gamma_{\O^\m} = \{ \gamma \in \iota_\infty(\O^\m), \ \det(\gamma)=1\}$$ on $\H$.  In order to prove Theorem \ref{t:maingen}, it suffices to prove that  \begin{equation}\label{e:toprove}|\phi'(g)| \ll_{T, \eps}C'(\pi)^{O(1)} C_1(\pi)^{\frac{\delta}{2} + \frac{\eta_1}{2} - \frac{\eta_2}6 +\eps} \end{equation} for all $g = \prod_v g_v \in G(\A)$ satisfying \begin{equation}\label{e:toprove2}
g_p \in K_p  \text{ for all } p \in \f, \quad \det(\iota_\infty(g_\infty))>0, \quad \text{ and }  \iota_\infty(g_\infty)(i) \in \J.
\end{equation}  This is because any element of $G(\A)$ can be left-multiplied by a suitable element of $Z(\A)G(\Q)$ so that $g$ has the above property.

The rest of this subsection is devoted to proving \eqref{e:toprove}.

\subsubsection*{Test functions}

We define a test function $\kappa$ on $G(\A)$, which will be essentially the same as the one used in \cite{saha-sup-minimal}.
Let $S = \ram(\pi) \cup \{p \in \f: p|d\}$.
Let $\ur = \f \setminus S$ be the set of primes not in $S$. We will choose $\kappa$ of the form $\kappa = \kappa_S \kappa_\ur  \kappa_\infty$. For convenience, we denote $G_S = \prod_{p \in S}G_p$, $\Q_S^\times = \prod_{p \in S} \Q_p^\times$, and $\O_S^\times = \prod_{p\in S}\O_p^\times$. By assumption, the action of $\O_S^\times$ on $\phi'$ generates an irreducible representation of dimension $\ll_{\eps} C_1(\pi)^{\delta+ \eps}$.

We define the function $\kappa_S$ on $G_S$ as follows:
$$\kappa_S (g_S)=  \begin{cases}  0 &\text{ if } g_S \notin  \Q_S^\times\O_S^\times, \\ \omega_{\pi}^{-1}(z) \langle \phi', \pi(k) \phi' \rangle &\text{ if }  g_S=  zk, \quad z \in \Q_S^\times, \ k \in  \O_S^\times. \end{cases}$$

Then  as in Section 4.1 of \cite{saha-sup-minimal}, we have  \begin{equation}\label{kappaN}R(\kappa_S)\phi':=\int_{\Q_S^\times \bs G_S} \kappa_S(g)(\pi(g)\phi')  \ dg = \lambda_S \phi', \quad \text{ where } \lambda_S \gg_{\eps} \frac{1}{NC_1(\pi)^{\delta +\eps}}. \end{equation}

Next we move on to the primes in $\ur$. We define $\kappa_\ur$ exactly as in Section 4.1 of \cite{saha-sup-minimal}. The definition of $\kappa_\ur$ depends on a parameter $\Lambda$ that we will fix later. As shown in \cite{saha-sup-minimal}, \begin{equation}\label{kappaur}R(\kappa_\ur )\phi' = \lambda_\ur \phi', \quad \lambda_\ur \gg_\eps \Lambda^{2-\eps}.\end{equation}

Finally, we consider the infinite place. As we are not looking for a
bound in the archimedean aspect, the choice of $\kappa_\infty$ is
unimportant. However for definiteness, let us fix the  function
$\kappa_\infty$ as follows. Let $f:\R_{\ge 0} \rightarrow [0,1]$ be a smooth non-increasing function such that $f(x)=1$ if $x \in [0, \frac12]$ and $f(x)=0$ if $x \ge 1$. Let $g \in \GL_2(\R)^+$, and define $u(g) = \frac{|g(i) - i|^2}{4 \Im(g(i))}$. Define $$\kappa_\infty(g) = f(u(g)) \langle \phi' , \pi(g) \phi'\rangle,  $$ for $g \in  \GL_2(\R)^+$ and define $\kappa_\infty$ to be equal to identically zero on $\GL_2(\R)^-$. Then we have that $$\kappa_\infty(g) \neq 0 \Rightarrow
\det(\iota_\infty(g_\infty))>0, \ u(\iota_\infty(g_\infty)) \le 1$$ and furthermore the operator $R(\kappa_\infty)$ satisfies
\begin{equation}\label{kappainf}R(\kappa_\infty)\phi' =
\lambda_\infty \phi', \quad \lambda_\infty \gg_{T} 1.\end{equation}

 We define the automorphic kernel $K_\kappa(g_1, g_2)$ for $g_1, g_2 \in G(\A)$
via $$K_\kappa(g_1, g_2) = \sum_{\gamma \in  G'(\Q)} \kappa(g_1^{-1} \gamma
g_2).$$

Now, as in Section 4.2 of \cite{saha-sup-minimal}, we get

\begin{equation}\label{keyineq1} |\phi'(g)|^2 \ll_{T, \eps} NC_1(\pi)^{\delta + \eps}\Lambda^{-2 + \eps}   K_\kappa(g,g).\end{equation}

On the other hand, we have by construction \begin{equation}\label{keyineqha}K_\kappa(g,g) \le \sum_{1 \le \ell \le 16\Lambda^4} \frac{y_\ell}{\ell^{1/2}}\sum_{\substack{\gamma \in G'(\Q)\\ \kappa_\ell(\gamma) \neq 0 \\ \kappa_\infty(g_\infty^{-1}\gamma_\infty g_\infty) \neq 0} } \left|\kappa_S(g_S^{-1}\gamma_S g_S) \right|,\end{equation} where the $y_\ell$ satisfy \begin{equation}\label{e:ylcases}|y_\ell| \ll \begin{cases}\Lambda, & \ell=1,\\1, &\ell = \ell_1 \ell_2 \text{ or } \ell = \ell_1^2\ell_2^2 \text{ with } \ell_1, \ell_2 \in \P, \\ 0, &\text{otherwise,}\end{cases}\end{equation} with $\P= \{\ell: \ell \text{ prime, } \ell \in \ur, \ \Lambda \le \ell \le 2\Lambda\}$ and where $\kappa_\ell= \prod_{p \in \ur}\kappa_{\ell, p}$ is a function on $\prod_{p \in \ur} G(\Q_p)$ that is defined in Section 3.5 of \cite{saha-sup-level-hybrid} (see also Section 4.1 of \cite{saha-sup-minimal}); we recall that $\kappa_{\ell, p}$ is supported on $\Q_p^\times \O_p(\ell)$ where $\O_p(\ell) = \{ \alpha \in \O_p: \nr(\alpha) \in \ell \Z_p^\times\}$.

Let us look at \eqref{keyineqha} more carefully. First of all, note that if $\kappa_\ell(\gamma) \kappa_\infty(g_\infty^{-1}\gamma_\infty g_\infty) \neq 0$ then

\begin{enumerate}
\item [(a)] $\gamma_p \in \Q_p^\times \O_p(\ell) \quad \forall p \in \ur$,
\item [(b)] $\det(\iota_\infty(\gamma_\infty))>0$, $u(z, \iota_\infty(\gamma_\infty) z) \le 1$, where $z = g_\infty i$.
    \end{enumerate}

 Looking at the primes $p |d$ we see that $\kappa_S(g_S^{-1}\gamma_S g_S) \neq 0$ implies that

\begin{enumerate}
\item [(c)] $\gamma_p \in \Q_p^\times \O_p^\times \quad \forall p |d$.
    \end{enumerate}

(We remind the reader here that $\O_p = \O^\m_p$ if $p \in \ur$, or if $p|d$.)

Consider the primes $p \in \ram(\pi)$. If $\kappa_p(g_p^{-1}\gamma_p g_p) \neq 0$, then clearly $g_p^{-1}\gamma_p g_p \in \Q_p^\times \O_p^\times$, or equivalently, $\gamma_p \in \Q_p^\times ({}^g\O)_p^\times$. So far, we have not at all used condition $(2)$ of Definition \ref{d:uniformlycontrolled}. We now do so. For each prime $p \in \ram(\pi)$ define $r_p =  a_1(\pi_p)+1 .$ Define $R_p = \{1, \ldots , r_p\}$ and let $R$ be the set-theoretic product  $\prod_{p \in \ram(\pi)}R_p$. For each $u = (u_p)_{p\in \ram(\pi)} \in R$, where each $u_p \in R_p$, associate another tuple $H_u = (\eta_{p,u_p})_{p\in \ram(\pi)}$ as follows: $\eta_{p, 1} = \eta_1$ and $\eta_{p,i} = \eta_1 + (i-1) \frac{\eta_2-\eta_1}{a_1(\pi_p)}$ for all $1\le i \le r_p$.

Now consider a $\gamma \in G'(\Q)$ which satisfies (a)-(c) above and such that $\gamma_p \in  \Q_p^\times ({}^g\O)_p^\times$ for each $p \in \ram(\pi)$. It is clear that for any such $\gamma$, there exists a unique tuple $u \in R$ such that

\begin{enumerate}
\item [(d)] $g_p^{-1}\gamma_p g_p \in \Q_p^\times (\L^{\eta_{p,u_p}}_p \cap \O_p^\times)$,  $g_p^{-1}\gamma_p g_p  \notin \Q_p^\times (\L^{\eta_{p,u_p+1}}_p \cap \O_p^\times) \quad \forall p \in \ram(\pi)$.
    \end{enumerate}

Above, we adopt the convention that $\L_p^{\eta_{p,r_p+1}}$ is the \emph{empty set} for each $p \in \ram(\pi)$, so that the second part of condition (d) is automatic for the primes where $u_p=r_p$.

It is clear from the above discussion that the contribution to the right-most sum in \eqref{keyineqha} only come from those $\gamma$ for which the conditions (a)-(d) above are satisfied for \emph{some} tuple $u \in R$. Furthermore, whenever the conditions (a)-(d) above are satisfied for a particular $u$, condition 2(d) of Definition \ref{d:uniformlycontrolled} implies that $$\left|\kappa_S(g_S^{-1}\gamma_S g_S) \right| \ll C'(\pi)^{O(1)} \prod_{p\in \ram(\pi)}p^{(\eta_2 - \eta_1)(u_p  - a_1(\pi_p))}.$$

For each tuple $u$, recall the definition of the lattice $\L^{H_u}$, which is precisely the global lattice corresponding to the collection of local lattices $\{\L_p^{\eta_{p,u_p}}\}_{p\in \ram(\pi)}$.  Define $${}^g\L^{H_u}(\ell;z, 1)  =\{\alpha \in {}^g\L^{H_u}: \nr(\alpha) = \ell, u(z, \iota_\infty(\alpha) z) \le 1\}.$$
By Proposition 4.2 of \cite{saha-sup-minimal}, the number of $\gamma \in G'(\Q)$ satisfying (a)-(d) above is bounded by the size of $|{}^g\L^{H_u}(\ell;z, 1)|$.

Therefore, we conclude \begin{equation}\label{keyineqha2}K_\kappa(g,g) \ll C'(\pi)^{O(1)} \sum_{u\in R} \sum_{1 \le \ell \le 16\Lambda^4} \frac{y_\ell}{\ell^{1/2}}|{}^g\L^{H_u}(\ell;z, 1)|  \prod_{p\in \ram(\pi)}p^{(\eta_2 - \eta_1)(u_p  - a_1(\pi_p))}.\end{equation}

Now, using the fact that the lattice ${}^g\L^{H_u}$ is tidy in $\O^\m$ and has index $N^{H_u}$ in $\O^\m$, we use Proposition \ref{t:counting} and \eqref{e:levelbds} to obtain for each $1 \le L \le C(\pi)^{O(1)}$:
\begin{equation}\label{eq1}\sum_{1\le m \le L } |{}^g\L^{H_u}(m;z, 1)| \ll_{\epsilon} C(\pi)^{\epsilon} \left(  L +C'(\pi)^{O(1)}\frac{L^2}{N\prod_{p \in \ram(\pi)} p^{(\eta_2 - \eta_1)(u_p  - 1)} } \right),\end{equation}
 \begin{equation}\label{eq2}\sum_{1\le m \le L } |{}^g\L^{H_u}(m^2;z, 1)| \ll_{\epsilon}C(\pi)^{\epsilon} \left(  L +C'(\pi)^{O(1)}\frac{L^3}{N\prod_{p \in \ram(\pi)} p^{(\eta_2 - \eta_1)(u_p  - 1)} } \right).\end{equation}

Combining \eqref{e:ylcases}, \eqref{keyineqha2}, \eqref{eq1}, \eqref{eq2}, we get

\begin{equation}\label{keyineqha3}\begin{split}K_\kappa(g,g) &\ll_{ \eps} C'(\pi)^{O(1)}C(\pi)^{\epsilon} \left(\Lambda + \frac{\Lambda^4}{NC_1(\pi)^{\eta_2-\eta_1}}\right)  \sum_{u\in R } 1 \\ &\ll_\eps C'(\pi)^{O(1)} C(\pi)^\eps \left(\Lambda + \frac{\Lambda^4}{NC_1(\pi)^{\eta_2-\eta_1}}\right)  \end{split}\end{equation}
 since $|R| \ll_\eps C(\pi)^\eps$.

From  \eqref{keyineq1} and \eqref{keyineqha3} we  obtain the pivotal inequality:

\begin{equation}\label{keyineq2} |\phi'(g)|^2 \ll_{T, \eps}  C_1(\pi)^{\eta_1+\delta + \eps}C'(\pi)^{O(1)}\left( \frac{1}{\Lambda} + \frac{\Lambda^2}{C_1(\pi)^{\eta_2} } \right).   \end{equation}

Now, putting $\Lambda = C_1(\pi)^{\frac{\eta_2}3}$, we immediately obtain \eqref{e:toprove}, as required.

\section{Some $p$-adic stationary phase analysis}\label{s:local}
This section will be purely local. The results of this section will complete the proof of  Proposition \ref{t:mainlocal}.
\subsection{Notations}The following notations will be used throughout Section \ref{s:local}. We let $F$ be a non-archimedean local
  field  of characteristic zero. We assume throughout that $F$ has \emph{odd} residue cardinality $q$.
Let $\mathfrak{o}$ be its ring of integers,
and $\mathfrak{p}$ its maximal ideal.
Fix a uniformizer $\varpi$ of $\OF$ (a choice of generator of $\mathfrak{p}$) .
Let $|.|$ denote the absolute value
on $F$ normalized so that
$|\varpi| = q^{-1}$. For each $x \in F^\times$, let $v(x)$ denote the integer such that $|x| = q^{-v(x)}$.  For a non-negative integer $m$, we define the subgroup $U_m$ of  $\OF^\times$ to be the set of elements $x \in \OF^\times$ such that $v(x-1) \ge m$.

Let $\psi$ be a fixed non-trivial additive character of $F$, and let $a(\psi)$ be the smallest integer such that $\psi$ is trivial on $\p^{a(\psi)}$. For $\chi$  a multiplicative character of $F$, let $a(\chi)$ be the smallest integer such that $\chi$ is trivial on $U_{a(\chi)}$. We recall the following well-known lemma (see, e.g., Lemma 2.37 of \cite{sahasupwhittaker}).
\begin{lemma}\label{Lem:alphaofchi}
Let $\chi$ be a multiplicative character over  $F$ with $a(\chi)\geq 2$. Then there exists $\alpha_\chi\in F^\times$ such that $v(\alpha_\chi)=-a(\chi)+a(\psi)$ and
\begin{equation}
\chi(1+\D x)=\psi(\alpha_\chi \D x)
\end{equation}
for any $\D x\in \p^{\lceil a(\chi)/2\rceil}$.
\end{lemma}

Throughout this section, we denote $\O = M_2(\OF)$, $G = \GL_2(F)$ and $K=\GL_2(\OF)$. Define subgroups
$N =
\{n(x):  x\in F \}$,
$A = \{a(y): y\in F^\times \}$,
$Z =\{ z(t):
t \in F^\times \}$, $B_1=NA$,
and $B = Z N A = G \cap
\left[
  \begin{smallmatrix}
    *&*\\
    &*
  \end{smallmatrix}
\right]$ of $G$.
 For each non-negative integer $r$,$s$
denote
\[
K_0(r)
= K \cap \begin{pmatrix}
  * & *  \\
  \p^r  & *
\end{pmatrix},   K^\ast(r,s) = K \cap \begin{pmatrix}
  \ast  & \p^s  \\
  \p^r  & \ast
\end{pmatrix},  \O(r) =  \O \cap \begin{pmatrix}
  \ast  & \p^r  \\
  \p^r  & \ast
\end{pmatrix},  K^\ast(r)= K^\ast(r,r) = (\O(r))^\times.
\]

We note our normalization of Haar measures.
The measure $dx$ on the additive group $F$ assigns volume 1
to $\OF$, and transports to a measure on $N$.
The measure $d^\times y$ on the multiplicative group $F^\times$ assigns
volume 1 to $\OF^\times$,
and transports to measures on
$A$ and $Z$.
We obtain a left Haar measure $d_Lb$ on $B$ via
$d_L(z(u)n(x)a(y)) = |y|^{-1}\, d^\times u \, d x \, d^\times
y.$
Let $dk$ be the probability Haar measure on $K$.
The Iwasawa decomposition
$G = B K$ gives a left Haar measure $dg = d_L b \, d k$ on $G$.

Let $\pi$ be an irreducible, infinite-dimensional, unitary representation of $G$ with trivial central character. We define $a(\pi)$ to be the smallest non-negative integer such that $\pi$ has a $K_0(a(\pi))$-fixed  vector. Let $\langle, \rangle$
denote a  $G$-invariant inner product on $V_\pi$ (which is unique up to multiples).

We will use the following notation:
\begin{itemize}
\item $n =a(\pi)$,
\item $n_1 := \lceil \frac{n}{2} \rceil$,

\item $n_0:= n -n_1 = \lfloor \frac{n}{2} \rfloor$.
\end{itemize}

We let $v_\pi$ denote a \emph{newform} in the space of $\pi$, i.e., a non-zero vector fixed by $K_0(\p^n)$; it is known that $v_\pi$ is unique up to multiples. Put $v_\pi' = \pi(a(\varpi^{n_1}))v_\pi$. Note that $v_\pi'$ is  the unique (up to multiples) non-zero vector in $\pi$ that is invariant under the subgroup $a(\varpi^{n_1})K_0(n)a(\varpi^{-n_1})$. Define matrix coefficients $\Phi_\pi$, $\Phi_\pi'$ on $G$ as follows:

$$\Phi_\pi(g) =
\frac{\langle
v_\pi , \pi(g) v_\pi\rangle}{\langle v_\pi , v_\pi\rangle},$$

$$\Phi_\pi'(g) =
\Phi_\pi(a(\varpi^{-n_1})ga(\varpi^{n_1})) = \frac{\langle
v_\pi' , \pi(g) v_\pi'\rangle}{\langle v_\pi' , v_\pi'\rangle}.$$These definitions are independent of the choice of $v_\pi$ or of the inner product.

\subsection{A reformulation of Proposition \ref{t:mainlocal}}
For the rest of Section \ref{s:local}, let $\pi$, $v_\pi$, $v_\pi'$, $\Phi_\pi'$ be as above, and assume that $a(\pi)>2$ and $\pi$ has trivial central character. This is sufficient for the purpose of proving Theorem \ref{t:mainlocal}, as noted in Remark \ref{rem:localsmallexclude}.

\begin{proposition}\label{prop:mainlocal}  For each representation $\pi$ as above, the following hold:
\begin{enumerate}
\item [(a)]  The subrepresentation of $\pi|_{K^\ast(1)}$ generated by $v_\pi'$ is irreducible of dimension $\ll q^{n_0}$.

\item  [(b)] Let $j \le n_1$. Then for all $g \in     K^\ast(1)$, $g \notin  K^\ast(j+1)$, we have $|\Phi_\pi'(g)| \ll q^{\frac{j-n_1}{2} + O(1)}.$

\end{enumerate}
\end{proposition}

Before starting on the proof of Proposition \ref{prop:mainlocal}, we explain how it implies Proposition \ref{t:mainlocal}.
\begin{proof}[Proof that Proposition \ref{prop:mainlocal} implies Proposition \ref{t:mainlocal}] Let $\eta_1 = 0$, $\eta_2 =1/2$, $\delta=1$.  Let $p$ be an odd prime not dividing $d$, and consider Proposition \ref{prop:mainlocal} with $F=\Q_p$. We need to show that the conditions $(1), (2)$ of Definition \ref{d:uniformlycontrolled} hold. In the context of Definition \ref{d:uniformlycontrolled} $\pi_{i,p} = \pi$, $v_{i,p}=v_\pi$ where $\pi$, $v_\pi$ are as defined in the beginning of this section. We define $g_{i,p} = \iota_p^{-1}\mat{\varpi^{a_1(\pi_i)}}{}{}{1}$, and $\O_{i,p} = \iota_p^{-1}(\O(1))$. The vector $v_{i,p}'$ from Definition  \ref{d:uniformlycontrolled} is then the vector $v_\pi'$ defined above. Now the condition (1) of Definition \ref{d:uniformlycontrolled} follows immediately from part (a) of Proposition \ref{prop:mainlocal}.

In order  to verify condition (2), let $0 \le \eta \le \frac12$. Define $j = \lfloor  n_1 \eta/2 \rfloor$ and put $\L^\eta_{i,p} = \iota_p^{-1}(\O(j+1))$.  Now condition (2) of Definition \ref{d:uniformlycontrolled} is an immediate consequence of part (b) of Proposition \ref{prop:mainlocal}.
\end{proof}

\begin{remark} For the purpose of verifying condition (2) in the proof above, we could have selected $j$ to be any non-decreasing integer valued function of $\eta \in [0, \frac12]$ satisfying $\frac{n_1 \eta}{2} - O(1) \le j \le 2n_1 \eta + O(1)$.
\end{remark}
\subsection{Proof of part (a) of Proposition \ref{prop:mainlocal}} Let us prove part (a) of Proposition \ref{prop:mainlocal}. Let  $V_1$ be the vector-space generated by the action of  $K^\ast(1)$ on $v_\pi'$. First we show that the action of $K^\ast(1)$ on $V_1$ is irreducible. If not, then there exists a direct sum decomposition $V_1 =  V_2 + V_3$ into non-zero subspaces
$V_2$ and $V_3$ which each admit an action of  $K^\ast(1)$. Since $v_\pi'$ generates $V_1$, its projections along $V_2$ and $V_3$ give two linearly independent vectors which are both fixed by the subgroup $a(\varpi^{n_1})K_0(n)a(\varpi^{-n_1}) \subseteq K^\ast(1)$ (recall that $a(\pi) \ge 2$). This contradicts newform theory, thus showing the irreducibility of $V_1$.

Next, we need to show that $\mathrm{dim}(V_1) \ll q^{n_1}$. Let  $V_2$ be the vector-space generated by the action of  $K^\ast(0,n_1-n_0)$ on $v_\pi'$. Since $K^\ast(1)$ is a subgroup of $K^\ast(0,n_1-n_0)$ it follows that $\mathrm{dim}(V_1) \le \mathrm{dim}(V_2)$. On the other hand Proposition 2.13 and Lemma 2.18 of \cite{saha-sup-level-hybrid} show that $\mathrm{dim}(V_2) \ll q^{n_0}.$ This completes the proof.

\subsection{A refinement of part (b)} In this subsection, we state a refinement of assertion (b) of Proposition \ref{prop:mainlocal} in terms of a Theorem that involves the matrix coefficient associated to the newvector.

\begin{theorem}\label{p:localnewmatrixmain}Let $y$, $z$ in $F^\times$ and $m \in F$. \begin{enumerate} \item Suppose that $n_0 <i <n-1$. Then we have \begin{equation}\label{e:1}\left|\Phi_\pi\left( \mat{y}{m}{0}{z}\mat{1}{0}{\varpi^i}{1}\right)\right| \ll q^{\frac{i-n}{2} + O(1)},\end{equation} and furthermore, for such $i$ as above, we have \begin{equation}\label{e:2}\Phi_\pi\left( \mat{y}{m}{0}{z}\mat{1}{0}{\varpi^i}{1}\right) \neq 0 \Rightarrow v(y)=v(z)=v(m)+n-i.\end{equation}

\item Suppose that $n-1 \le i \le n$. Then we have \begin{equation}\label{e:3}\Phi_\pi\left( \mat{y}{m}{0}{z}\mat{1}{0}{\varpi^i}{1}\right) \neq 0 \Rightarrow v(y)=v(z)\le v(m)+1.\end{equation}

    \end{enumerate}
\end{theorem}

Before starting on the proof of Theorem \ref{p:localnewmatrixmain}, we explain how it implies Proposition \ref{prop:mainlocal}.
\begin{proof}[Proof that Theorem \ref{p:localnewmatrixmain} implies Proposition \ref{prop:mainlocal}] Let $j$, $g$ be as in Proposition \ref{prop:mainlocal}. Since we have the trivial upper bound of 1 on $| \Phi_\pi'(h)|$ for all $h$, and since $g \in K^\ast(1)$ we may assume that $1\le j< n_0 -1$. Furthermore, by decreasing $j$ if necessary, we may assume that $g \in K^\ast(j)$. So putting $g = \mat{a}{b}{c}{d}$ we have $\min(v(b),v(c)) = j$.  Note that $\Phi_\pi'(g) = \Phi_\pi\left(\mat{a}{b'}{c'}{d}\right)$ where $c' = c\varpi^{n_1}$, $b' = b \varpi^{-n_1}$. We consider two cases. \medskip

\textbf{Case I:}  $v(c) = j$.

In this case we have $v(c') = n_1 + j$. Since $v(d)=0$, a direct calculation shows that $$\mat{a}{b'}{c'}{d} \in B(F)\mat{1}{0}{\varpi^{j+n_1}}{1}K_0(\p^n).$$ Therefore, \eqref{e:1} tells us that $\left|\Phi_\pi\left(\mat{a}{b'}{c'}{d}\right)\right|  \ll q^{\frac{j+n_1 - n}{2} + O(1)} \ll q^{\frac{j- n_1}{2}+ O(1)}$, as required.

\medskip

\textbf{Case II:}  $v(c) > j$.

In this case we have $v(b)=j$. As before we have $v(b') = j-n_1$, $v(c') = v(c)+n_1$, and $\Phi_\pi'(g) = \Phi_\pi\left(\mat{a}{b'}{c'}{d}\right)$. We can see from a direct calculation that $$\mat{a}{b'}{c'}{d} \in \mat{y}{m}{0}{z}\mat{1}{0}{\varpi^{r}}{1}K_0(\p^n)$$ for some $m \in F$, $y \in \OF^\times$, $z \in \OF^\times$ and $r = \min(n, v(c)+n_1)$. Note that $v(b')\ge v(m)$.

 We claim that $\Phi_\pi\left(\mat{a}{b'}{c'}{d}\right) =0$. Suppose not.  Suppose first that $v(c) < n_0-1$. Then $r=v(c)+n_1$, and  using \eqref{e:2} we see that $v(m)= v(c)-n_0$. This gives us $j-n_1=v(b')\ge v(m) =v(c)-n_0$, and hence that $v(c)\le j$, a contradiction. Next, suppose that $v(c)\ge n_0-1$. Then $n \ge r\ge n-1$ and using \eqref{e:3} we see that $j-n_1=v(b')\ge v(m) \ge  -1$.  So $j \ge n_1 -1$, which contradicts our earlier assumption that $j<n_0-1$.

\subsection{The proof of  Theorem \ref{p:localnewmatrixmain} }
The assertions \eqref{e:2} and \eqref{e:3} of Theorem \ref{p:localnewmatrixmain} have already been proven in \cite[Proposition 3.1]{Hu:17a}. So we only need to prove the upper bound part in Theorem \ref{p:localnewmatrixmain}, i.e., \eqref{e:1}.

 For simplicity denote
\begin{equation}
\Phii=\Phi_\pi\left(\zxz{a}{m}{0}{1}\zxz{1}{0}{\varpi^i}{1}\right).
\end{equation}
For the rest of this section, we fix an additive character $\psi$ of $F$ such that $a(\psi)=0$ and consider the Whittaker model of $\pi$ with respect to this character.  Using the usual inner product in the Whittaker model, it follows that,
\begin{equation}\label{Eq:WhitToMC}
\Phii=\int\limits_{v( x )=0}\psi(m x )W^{(i)}(a x )d^\times x ,
\end{equation}
where $W^{(i)}( x )=W_{\pi}\left(\zxz{ x }{0}{0}{1}\zxz{1}{0}{\varpi^i}{1}\right)$ and $W_\pi$ is the local Whittaker newform (see, e.g., Section 3 of \cite{Hu:17a} for more details).

The basic tool to analyze such integrals is the $p$-adic stationary phase analysis. Roughly speaking, we will rewrite this integral and break it up into pieces, and we will prove (using orthogonality of characters) that most of these pieces vanish. The required bounds will follow by counting the number of non-vanishing pieces.
 Since $n >2$ and $q$ is odd, there are two possibilities for $\pi$: principal series representations, and dihedral supercuspidal representations. We deal with each below. 

\subsubsection{Principal series representation}
Let $\pi=\pi(\mu_1,\mu_2)$ be a principal series representation. In this case $n$ is even and we take $\mu_2 = \mu_1^{-1} = \mu$, $a(\mu)=n_1=n_0=n/2$. Denote
\begin{equation}
C_0=\int\limits_{u\in \OF^\times} \mu( u)\psi(-\varpi^{-n_0} u)du.
\end{equation}
Using the usual intepretation as a Gauss sum (see, e.g., \cite[(6)]{sahasupwhittaker}) we see that  $|C_0|\asymp \frac{1}{q^{n_0/2}}.$

By \cite[Lemma 2.12]{hu_triple_2017},  we have
\begin{lemma}
When $n_0<i\leq n$, $W^{(i)}(x)$ is supported on $x \in \OF^\times$, and for $x \in \OF^\times$ we have
\begin{equation}
W^{(i)}( x )=C_0^{-1}\int\limits_{u\in \OF^\times} \mu(1+u\varpi^{i-n_0})\mu( x  u)\psi(-\varpi^{-n_0} x  u)du.
\end{equation}

\end{lemma}

Note that the condition $a(\pi)\geq 3$ implies that $a(\mu)\geq 2$.
Let $\AL$ be the constant associated to $\mu$ by Lemma \ref{Lem:alphaofchi}. Then $v(\AL)=-n_0$.

By the results of \cite{Hu:17a}, $\Phii$ is supported on $v(a)=0$ and $v(m)=i-n>-n_0$. Then by \eqref{Eq:WhitToMC},
\begin{align}\label{Eq:PhiforPrincipalseries}
\Phii&=C_0^{-1}\int\limits_{v(x)=0}\psi(mx)\int\limits_{u\in \OF^\times} \mu(1+u\varpi^{i-n_0})\mu(a x  u)\psi(-\varpi^{-n_0}a x  u)dud^\times  x \\
&=C_0^{-1}\iint\limits_{v(x)=v(u)=0} \psi(mxu)\mu(1+u^{-1}\varpi^{i-n_0})\mu(a x  )\psi(-\varpi^{-n_0}a x  )dud^\times  x.\notag
\end{align}

The idea is to break the above integral into small intervals, on each of which we can apply Lemma \ref{Lem:alphaofchi} to analyse the integral and get easy vanishing for most of the small intervals. This is the exact analogue of the archimedean  stationary phase analysis.
In the integrand in \eqref{Eq:PhiforPrincipalseries},
write $u=u_0(1+\D u)$  for $u_0\in \OF^\times/ (1+\p^{\lceil (n-i)/2\rceil})$, $\D u\in \p^{\lceil (n-i)/2\rceil} $,  and $x=x_0(1+\D x)$ for $x_0\in \OF^\times/(1+\p^{\lceil n_0/2\rceil})$, $\D x\in \p^{\lceil n_0/2\rceil}$. Using Lemma \ref{Lem:alphaofchi}  and the invariance properties of $\psi$ and $\mu$, we get
\begin{align}\label{e:matrixlocalkeyprinc}
&\Phii=C_0^{-1}\sum\limits_{x_0,u_0} \psi(mx_0u_0)\mu(1+u_0^{-1}\varpi^{i-n_0})\mu(a x_0  )\psi(-\varpi^{-n_0}a x_0 ) \\
& \quad \times \int\limits_{\p^{\lceil \frac{n-i}2\rceil} }  \int\limits_{\p^{\lceil \frac{n_0}2\rceil} } \psi\left(mu_0x_0\D x+mx_0u_0\D u -\AL\frac{\varpi^{i-n_0}u_0^{-1}}{1+\varpi^{i-n_0}u_0^{-1}} \D u+\AL \D x-\varpi^{-n_0}ax_0\D x\right)  d\D x d\D u     .\notag
\end{align}
For the innermost integral involving $\D x$, $\D u$ to be nonzero, we must have that
\begin{equation}\label{equation:principalseriesdx}
 mu_0x_0+\AL-\varpi^{-n_0}ax_0 \equiv 0 \mod \varpi^{-\lceil \frac{n_0}{2}\rceil},
\end{equation}
\begin{equation}\label{equation:principalseriesdu}
mx_0u_0 -\AL\frac{\varpi^{i-n_0}u_0^{-1}}{1+\varpi^{i-n_0}u_0^{-1}}\equiv 0\mod \varpi^{-\lceil\frac{n-i}{2}\rceil}.
\end{equation}
From the first equation, we get that
\begin{equation}\label{equation:principalseriesx0}
x_0\equiv -\frac{\AL}{mu_0-\varpi^{-n_0}a}\mod{\varpi^{\lfloor \frac{n_0}{2}\rfloor}}.
\end{equation}
So there is a unique $x_0$ $\mod \varpi^{\lfloor \frac{n_0}{2}\rfloor}$ for each $u_0\mod\varpi^{\lceil \frac{n-i}{2}\rceil}$ satisfying the above. As a trivial consequence, there are at most $q$ solutions of $x_0\mod\varpi^{\lceil\frac{n_0}{2}\rceil}$ for each $u_0\mod\varpi^{\lceil \frac{n-i}{2}\rceil}$.

Next by computing $ \eqref{equation:principalseriesdx}\times mu_0-\eqref{equation:principalseriesdu} \times  (mu_0-\varpi^{-n_0}a)$, we get the following necessary condition for non-vanishing:
\begin{equation}\label{equation:principalseriesu0}
\alpha\left(mu_0+\frac{\varpi^{i-n_0}u_0^{-1}(mu_0-\varpi^{-n_0}a)}{1+\varpi^{i-n_0}u_0^{-1}}\right)\equiv 0\mod \varpi^{-\lceil \frac{n-i}{2}\rceil -n_0}.
\end{equation}
Here we have used that $-\lceil \frac{n_0}{2}\rceil+i-n\geq -\lceil\frac{n-i}{2}\rceil-n_0$.
This congruence is equivalent to
\begin{equation}
 mu_0^2+2\varpi^{i-n_0}m u_0-\varpi^{i-n}a \equiv 0 \mod \varpi^{-\lceil\frac{n-i}{2}\rceil},
\end{equation}
as $v(\AL)=-n_0$. 
Note that $v(mu_0^2)=v(\varpi^{i-n}a)=i-n<v(2\varpi^{i-n_0}m u_0)$. So this quadratic equation is not degenerate when $p\neq 2$, and we can solve for at most two solutions of $u_0\mod\varpi^{\lfloor \frac{n-i}{2}\rfloor}$, and consequently at most $2q$ solutions of $u_0\mod\varpi^{\lceil \frac{n-i}{2}\rceil}$.

In summary we have that there are $\le 2q^2$ pairs $(x_0, u_0)$ contributing to \eqref{e:matrixlocalkeyprinc} and so we get
\begin{equation}
|\Phii|\ll |C_0^{-1}| 2q^2\Vol(\D x)\Vol(\D u)\asymp q^{\frac{i-n}{2}+ O(1)}.
\end{equation} as required.
\begin{remark}\label{rem:refined}
By going through the proof above more carefully (and looking at the cases $n_0$ odd and $n_0$ even) the implied constant in $O(1)$ in \eqref{e:1} can be worked out more explicitly. In particular when there are $O(q)$ solutions of $x_0$ and/or $u_0$, the sums in $x_0,u_0$ can be reduced to sums over the residue field and we expect complete square-root cancellation. The same comment applies to the supercuspidal representation case below.
\end{remark}

\subsubsection{Supercuspidal representations}
When $2\nmid q$, $\pi$ is associated by compact induction theory to a character $\theta$ over a quadratic field extension $\E/{F}$ with ramification index $e_\E$. Their relations are given explicitly as follows (see \cite{BH06})
\begin{enumerate}
\item $a(\pi)=n=2n_0$ corresponds to $e_\E=1$ and $a(\theta)=n_0$.
\item$n=2n_0+1$ corresponds to $e_\E=2$ and $a(\theta)=2n_0$.
\end{enumerate}
In the following we shall give uniform formulations and estimates  for both of these cases, which one can verify case by case according to this classification.
For simplicity, let $\E={F}(\sqrt{D})$ with $v_{F}(D)=e_\E-1$. We let $\OF_E$ denote the ring of integers of $E$, $\varpi_E$ denote a uniformizer of $E$ and $\p_E = \varpi_E\OF_E$.
Let $\psi_\E=\psi\circ\Tr_{\E/{F}}$. It's easy to check that $a(\psi_\E)=-e_\E+1$ since $a(\psi)=0$.
Let
\begin{equation}
C_0=\int\limits_{v_\E(u)=-a(\theta)-e_\E+1}\theta^{-1}(u)\psi_\E(u)d^\times u.
\end{equation}
Again by the usual interpretation as a Gauss sum, we get $|C_0|\asymp \frac{1}{q^{a(\pi)/2}}$.
Checking case by case, one can also see that for $u$ in the domain of the integral,
\begin{equation}
 v(N_{E/F}(u))=-n.
\end{equation}
The following lemma is a reformulation of \cite[Lemma 3.1]{assing18}.
\begin{lemma}
When $i>n_0$, $W^{(i)}(x)$ is supported on $v(x)=0$, and on the support,
\begin{equation}
W^{(i)}(x)=C_0^{-1}\int\limits_{v_\E(u)=-a(\theta)-e_\E+1}\theta^{-1}(u) \psi(-\frac{1}{x}\varpi^{i}N_{\E/{F}}(u))\psi_\E(u)d^\times u.
\end{equation}

\end{lemma}
 Again by \cite{Hu:17a} the matrix coefficient $\Phii$ is supported on $v(a)=0$, $v(m)=i-n$ when $n_0<i<n-1$. On the support, by the above lemma and \eqref{Eq:WhitToMC},
\begin{align}\label{Eq:PhiforSupercuspidal}
\Phii&=C_0^{-1}\int\limits_{v(x)=0}\psi(mx)\int\limits_{v_\E(u)=-a(\theta)-e_\E+1}\theta^{-1}(u) \psi(-\frac{1}{ax}\varpi^{i}N_{\E/{F}}(u))\psi_\E(u)d^\times ud^\times x\\
&=C_0^{-1}\int\limits_{v(x)=0}\psi(m\frac{1}{x})\int\limits_{v_\E(u)=-a(\theta)-e_\E+1}\theta^{-1}(u) \psi(-\frac{x}{a}\varpi^{i}N_{\E/{F}}(u))\psi_\E(u)d^\times ud^\times x\notag.
\end{align}

Since $a(\pi)\geq 3$, we have $a(\theta)\geq 2$. Let $\AL\in \E^\times$ be the constant associated to $\theta$ by Lemma \ref{Lem:alphaofchi}, then $v_\E(\AL)=-a(\theta)+a(\psi_\E)=-a(\theta)-e_\E+1$. As $\theta|_{{F}^\times}$ is essentially the central character $w_\pi$ which is trivial, we can assume that $\AL$ is purely imaginary in $\E^\times$.
In the integrand in \eqref{Eq:PhiforSupercuspidal}, write $x=x_0(1+\D x)$ with $x_0\in \OF^\times /(1+\p^{\lceil (n-i)/2\rceil})$, $\D x\in \p^{\lceil (n-i)/2\rceil}$, and $$u=u_0(1+\D u)=(a_0+\sqrt{D} b_0)(1+\D a+\sqrt{D}\D b)$$ for $u_0\in (\varpi_E^{-a(\theta)-e_\E+1}\OF_E/\varpi_E^{-\lfloor a(\theta)/2\rfloor-e_\E+1}\OF_E)^\times$, $\D u = \D a+\sqrt{D}\D b\in \varpi_\E^{\lceil a(\theta)/2\rceil} \OF_\E$. 

Then
\begin{equation}\label{e:intphiii}
\Phii=C_0^{-1}\sum\limits_{x_0,u_0}\psi(\frac{m}{x_0})\theta^{-1}(u_0)\psi(-\frac{x_0}{a}\varpi^i N_{\E/{F}}(u_0))\psi_\E(u_0) \iint\limits_{\substack{\D x \in \p^{\lceil \frac{n-i}2\rceil}  \\ \D u \in \p_\E^{\lceil \frac{a(\theta)}2\rceil} }}F_{x_0, u_0}(\D x, \D u) \, d\D x d\D u \end{equation} where $F_{x_0, u_0}(\D x, \D u)=
 \psi\left(-\frac{m}{x_0}\D x-2\AL \sqrt{D}\D b-\frac{x_0}{a}\varpi^i N_{\E/F}(u_0)(\D x - 2 \D a)+2a_0\D a+2Db_0\D b \right)$ for $\D u= \D a+\sqrt{D}\D b$ with $\D a$, $\D b$ in $\OF$.
Here we have used that $$\theta^{-1}(1+\D u)=\psi_E(-\alpha \D u)=\psi_E(-\alpha \D a-\alpha\sqrt{D}\D b)=\psi(-2\alpha\sqrt{D}\D b),$$
$$\psi_E(u_0\D u)=\psi(2a_0\D a+2Db_0\D b                       ).$$

Using the fact that $a(\psi)=0$, we observe that in order for the integral in \eqref{e:intphiii} to be nonzero, we need the following conditions to hold: \emph{For all $x_1$, $a_1$, $b_1$ in $\OF$ such that $x_1 \in \p^{\lceil \frac{n-i}2\rceil}$, $a_1+\sqrt{D}b_1 \in \p_\E^{\lceil \frac{a(\theta)}2\rceil}$, we have }
\begin{align}
&(\frac{m}{x_0}+\frac{x_0}{a}\varpi^i N_{\E/F}(u_0)) \, x_1 \in \OF,   \label{equation:sc1}   \\
&(a_0-\frac{x_0}{a}\varpi^i N_{\E/F}(u_0))\, a_1 \in \OF,   \label{equation:sc2} \\
&(Db_0-\alpha\sqrt{D})\, b_1 \in \OF.\label{equation:sc3}
\end{align}

Now, using a very similar analysis as in the principal series case, we shall see that the number of pairs $(x_0, u_0)$ satisfying \eqref{equation:sc1}, \eqref{equation:sc2}, and \eqref{equation:sc3} is $\ll q^{O(1)}.$

Consider the number of $b_0$ satisfying \eqref{equation:sc3} first. When $e_E=1$, or $e_\E=2$ and $a(\theta)/2$ is odd, we can choose $a_1$, $b_1$ in $\OF$ such that $ a_1+\sqrt{D}b_1\in \p_\E^{\lceil a(\theta)/2\rceil}$, $b_1 \in \varpi_\E^{\lceil a(\theta)/2\rceil-e_E+1}\OF_\E^\times \cap \OF$, which combined with \eqref{equation:sc3} gives us
\begin{equation}
b_0\equiv \frac{\alpha}{\sqrt{D}} \mod \p_\E^{-\lceil a(\theta)/2\rceil-e_E+1}
\end{equation}
while by the definition of $u_0$, $b_0\sqrt{D}$ is well defined up to $\p_\E^{-\lfloor a(\theta)/2\rfloor-e_E+1}$. Thus
\begin{equation}\label{Eq:Noofb0}
\sharp \{b_0 \text{\ satisfying \eqref{equation:sc3}}\}\ll q^{O(1)}.
\end{equation}

When $e_\E=2$ and $a(\theta)/2$ is even, we can choose $ b_1\in \varpi_\E^{\lceil a(\theta)/2\rceil}\OF_E^\times \cap \OF$, and this time  \eqref{equation:sc3} gives us $b_0\equiv \frac{\alpha}{\sqrt{D}} \mod \p_\E^{-\lceil a(\theta)/2\rceil-2e_E+2} $. By the same argument as above, \eqref{Eq:Noofb0} still holds in this case.

Similarly for each fixed $u_0$, there exists solutions for $x_0$ from \eqref{equation:sc1} iff $$-\frac{am}{\varpi^iN_{E/F}(u_0)}$$ is a square modulo $\varpi^{\lfloor (n-i)/2\rfloor}$. In that case we obtain
\begin{equation}\label{equation:x0}
x_0\equiv \pm\sqrt{-\frac{am}{\varpi^iN_{E/F}(u_0)}} \mod \varpi^{\lfloor (n-i)/2\rfloor}.
\end{equation}
Here we have used that $p\neq 2$.
So by the definition of $x_0$,
\begin{equation}
\sharp \{x_0 \text{\ satisfying \eqref{equation:sc1} for fixed $u_0$}\}\ll q^{O(1)}.
\end{equation}
Finally we come to counting $a_0$. When $e_\E=1$, or $e_\E=2$ and $a(\theta)/2$ is even, we can choose $ a_1$, $b_1$ so that $a_1\in \varpi_\E^{\lceil a(\theta)/2\rceil}\OF_\E^\times \cap \OF$, $ a_1+\sqrt{D} b_1\in \p_\E^{\lceil a(\theta)/2\rceil}$, so that from \eqref{equation:sc2} we now deduce
\begin{equation}\label{equation:sc2eq}
a_0-\frac{x_0}{a}\varpi^iN_{E/F}(u_0)\equiv 0 \mod \varpi_E^{-\lceil a(\theta)/2\rceil}.
\end{equation}
Note that if $v_E(\frac{x_0}{a}\varpi^iN_{E/F}(u_0))=e_E(i-n)\geq -\lceil a(\theta)/2\rceil$, we get a unique solution $a_0\equiv 0\mod \varpi_E^{-\lceil a(\theta)/2\rceil}$, and by the definition of $a_0$ and the previous results,
\begin{equation}\label{equation:sctotalcount}
\sharp \{(a_0, b_0, u_0) \text{\ satisfying \eqref{equation:sc1}\eqref{equation:sc2}\eqref{equation:sc3}}\}\ll q^{O(1)}.
\end{equation}

Otherwise when $e_E(i-n)<-\lceil a(\theta)/2\rceil$, \eqref{equation:sc2eq} is a nontrivial congruence relation and $v(a_0)=i-n$. As $p\neq 2$, we have for any solution $a_0$,
$$
 a_0+\frac{x_0}{a}\varpi^iN_{E/F}(u_0)\equiv 0\mod \varpi^{i-n}.
$$
Multiplying it with \eqref{equation:sc2eq} and substituting \eqref{equation:x0}, we get
\begin{equation}
a_0^2\equiv -\frac{m}{a }\varpi^iN_{E/F}(u_0)=-\frac{m}{a }\varpi^i(a_0^2-b_0^2D) \mod \varpi_E^{-\lceil a(\theta)/2\rceil} \varpi^{i-n}.
\end{equation}
One can get at most two solutions of $a_0\mod \varpi_E^{-\lceil a(\theta)/2\rceil}$ for each fixed $b_0$. So \eqref{equation:sctotalcount} is still true. 

If $e_\E=2$ and $a(\theta)/2$ is odd, we can instead choose $a_1\in \varpi_\E^{\lceil a(\theta)/2\rceil+1}\OF^\times_\E \cap \OF$ in the argument above \eqref{equation:sc2eq}. The rest of the discussions are similar and \eqref{equation:sctotalcount} still holds.

In conclusion we get that
\begin{equation}
|\Phii|\ll q^{O(1)} |C_0^{-1}|\Vol(\D x)\Vol(\D u)\asymp \frac{1}{q^{(n-i)/2 + O(1)}},
\end{equation}
as required.
\end{proof}

\section{An application to subconvexity}
In this Section we explain how Corollary \ref{c:mainnewforms} leads to a subconvexity result for certain central $L$-values.

\subsection{The setup and the main result}
Throughout this section, we will go back to the global setting and freely use the notations defined in Section \ref{s:globalstatement} and Section \ref{s:globalfamilies}. Recall that we have fixed an indefinite quaternion division algebra $D$ over $\Q$ of discriminant $d$. In addition for this section we fix: \begin{itemize}

\item A squarefree integer $P$ such that $(P, 2d)=1$.

\item A quadratic number field $K/\Q$ such that \begin{itemize}
    \item All primes dividing $P$ are split in $K$,
    \item All primes dividing $d$ are inert in $K$.
    \end{itemize}

    \end{itemize}

Let $\mathcal{S}_{P}$ denote the set of irreducible, unitary, cuspidal, automorphic representations $\pi = \otimes_v \pi_v$ of $G(\A)$ with the following properties:
\begin{enumerate}
\item $\pi$ has trivial central character.

\item If $\ell$ is a prime such that $\ell\nmid P$, then $\pi_\ell$ is spherical (i.e., has a non-zero $K_\ell$-fixed vector).

\end{enumerate}

Note that (using the notation of Section \ref{s:globalfamilies}), for any $\pi \in \mathcal{S}_{P}$, we have $C'(\pi)$ divides $P$, and hence $C(\pi)$ is divisible only by primes dividing $P$. We remind the reader that $C(\pi)$ denotes the ``away-from-$d$-part" of the conductor of $\pi$ (the conductor of $\pi$ equals $dC(\pi)$). We let $\O_K$ denote the ring of integers of $K$ and $\rho_K$ the quadratic character on $\Q^\times \bs \A^\times$ associated to the extension $K/\Q$.

\begin{remark}By the Jacquet-Langlands correspondence, the set $\mathcal{S}_P$ is in \emph{functorial} bijection with the set of irreducible, unitary, cuspidal, automorphic representations on $\rm{PGL_2}(\A)$ whose conductor equals $dC$ for some $C|P^\infty$. \end{remark}

 Given $\pi \in \mathcal{S}_{P}$ and a character $\chi$ of $K^\times \bs \A_K^\times$ such that $\chi |_{\A^\times} = 1$, we are interested in the central $L$-value $L(1/2 , \pi \times \AI(\chi))$ of the Rankin-Selberg $L$-function. Here $\AI(\chi)$ denotes the global automorphic induction of $\chi$ from $\A_K^\times$ to $\GL_2(\A)$, whose existence follows either from the converse theorem (see Chapter 7 of \cite{Gelbart1975}) or more explicitly via the theta correspondence \cite{shalika-tanaka}. By purely local calculations  \cite[(a2)]{Rohrlich1994}, it can be seen that the conductor of $\AI(\chi)$ equals $\disc(K)  N(\cond(\chi))$.

 \begin{theorem}\label{t:sub}Let $P$, $K$ and $\mathcal{S}_P$ be as above. Let $\chi$ be a character of $K^\times \bs \A_K^\times$ such that $\chi |_{\A^\times} = 1$ and such that $\gcd(C(\chi), d)=1$ where $C(\chi)= N(\cond(\chi))$ equals the absolute norm of the conductor of $\chi$. Then for any $\pi \in \mathcal{S}_P$, we have $$L(1/2, \pi \times \AI(\chi)) \ll_{K, P, \pi_\infty, \chi_\infty, \eps} C(\pi)^{5/12+\eps} C(\chi)^{1/2+\eps}.$$
 \end{theorem}

 The above theorem immediately implies a subconvexity result for $L(1/2, \pi \times \AI(\chi))$ for \emph{fixed} $\chi$ and varying $\pi \in \mathcal{S}_P$.
 \begin{corollary}\label{c:sub}Let $P$, $K$, $\chi$ and $\mathcal{S}_P$ be as in Theorem \ref{t:sub}. Then for $\pi \in \mathcal{S}_P$, we have $$L(1/2, \pi \times \AI(\chi)) \ll_{K, P, \pi_\infty, \chi, \eps} \big(C(\pi \times \AI(\chi))\big)^{5/24+\eps}$$ where $C(\pi \times \AI(\chi))$ denotes the (finite part of the) analytic conductor of $L(s, \pi \times \AI(\chi))$.
 \end{corollary}
\begin{proof} Any ``conductor dropping" for  $\pi \times \AI(\chi)$ is only potentially possible at primes $p|P$ for which $v_p(C(\chi)) = v_p(C(\pi))>0$. More precisely, let $\P_1$ be the set of prime numbers $p$ such that  $p | C(\chi)$ and $v_p(C(\chi)) = v_p(C(\pi))$. Then using Proposition 3.4 of \cite{tunnell78}, we see that $$C(\pi \times \AI(\chi)) = d^2\disc(K)^2\frac{\mathrm{lcm}(C(\pi)^2, C(\chi)^2)}{\prod_{p \in \P_1} p^{t_p}}$$ where the $t_p$ are non-negative integers satisfying $t_p \le 2v_p(C(\chi))$.  It follows immediately that  \begin{equation}\label{e:triv}C(\pi \times \AI(\chi)) \gg_{\chi} C(\pi)^2.\end{equation} The desired result follows from \eqref{e:triv} and Theorem \ref{t:sub}.
\end{proof}

 \begin{remark}By definition, $L(s , \pi \times \AI(\chi))$ is the finite part of the Langlands $L$-function attached to the automorphic representation $\pi \boxtimes \AI(\chi)$ on $D^\times \times \GL_2$. It is immediate that $L(s, \pi \times \AI(\chi)) = L(s, \pi' \times \AI(\chi))$ where $\pi'$ is the automorphic representation on $\GL_2(\A)$ associated to $\pi$ via the Jacquet-Langlands correspondence. Hence $L(s, \pi \times \AI(\chi))$ can be viewed as an $L$-function on $\GL_2(\A) \times \GL_2(\A)$.

 We remark that $G'=PD^\times$ is isomorphic to an orthogonal group $\SO(V)$ where $V$ is a three dimensional quadratic space. So $\pi$ can be regarded as an automorphic representation of $\SO(V)$. Moreover, $\Q^\times \bs K^\times \simeq \SO(W)$, where $W\subset V$ is a two dimensional quadratic space; this allows us to view $\chi$ as an automorphic representation $\pi_0$ on $\SO(W)$. Under this viewpoint, $L(s, \pi \times \AI(\chi)) = L(s, \pi \boxtimes \pi_0)$ is the standard $L$-function on $\SO(V) \times \SO(W)$, which puts it into the Gross-Prasad framework.

   Finally, we have that $$L(s , \pi \times \AI(\chi)) = L(s, \pi_K \times \chi)$$ where $\pi_K$ denotes the base-change of $\pi$ to $G(\A_K)$. Thus $L(s , \pi \times \AI(\chi))$ can be also viewed as an $L$-function on  $D^\times(\A_K) \times \A_K^\times$ or on $\GL_2(\A_K) \times \A_K^\times$.

   Thus, Theorem \ref{t:sub} can be regarded as a subconvexity result for any of the groups $G(\A) \times \GL_2(\A)$,  $\GL_2(\A) \times \GL_2(\A)$, $\SO(V)(\A)\times \SO(W)(\A)$, $G(\A_K) \times \GL_1(\A_K)$, and $\GL_2(\A_K) \times \GL_1(\A_K)$.     We also note that if $\chi=\bf 1$ is the trivial character, then $L(s , \pi \times \AI(\mathbf{1})) =  L(s , \pi)L(s, \pi \times \rho_K)$. In this special case, we suspect that other existing methods may give a superior exponent in the setting of Theoem \ref{t:sub}.
  \end{remark}

 \begin{remark} The representation $\AI(\chi)$ can be seen to be generated by the classical theta series (due to Hecke and Maass) associated to Hecke characters on $K^\times \bs \A_K^\times$. More precisely, we can identify a Hecke character $\chi$ on $K$ of  conductor $\mathbf{m}$ with a character on the group of fractional ideals of $K$ coprime to  $\mathbf{m}$. This allows us to write down explicitly an automorphic newform $\theta_\chi$ that generates $\AI(\chi)$. For example, suppose that $K=\Q(\sqrt{M})$ is an imaginary quadratic field with $M<0$ a fundamental discriminant. Suppose also that $\chi_\infty(\alpha) = \left(\frac{\alpha}{|\alpha|}\right)^\ell$ where $\ell\in \Z_{\ge0}$ and denote $Q = N(\mathbf{m})$. Then $\theta_\chi$ is the holomorphic newform\footnote{$\theta_\chi$ is a cusp form iff $\chi$ does not factor through the norm map; this happens if and only if $\chi^2 \neq 1$.} of weight $\ell+1$, level $|MQ|$ and character $\left(\frac{M}{\cdot}\right)$ given by the sum over ideals $\a$ as $$\theta_\chi(z) = \sum_{\a \subset \O_K} \chi(\a) (N(\a))^{\frac{\ell}2} e(N(\a)z).$$ We can write down a similar formula when $K$ is real; see Appendix A.1 of \cite{humphries-khan}. In this case the hypothesis $\chi|_{\A^\times} =1$ implies that $\theta_\chi$ is a weight 0 Maass form.
\end{remark}

For the convenience of the reader, we give a version of Theorem \ref{t:sub} that avoids any mention of quaternion algebras and that focusses on a single prime (``depth aspect") for simplicity.

\begin{corollary}\label{c:gl2sub}Let $p$ be an odd prime, and $d\neq 1$ a positive squarefree integer with an even number of prime factors. Assume that $(p,d)=1$. Let $M<0$ be a fundamental discriminant and put $K=\Q(\sqrt{M})$. Assume that $\left(\frac{M}{p}\right) = 1$ and $\left(\frac{M}{q}\right) = -1$ for all primes $q$ dividing $d$. Let $\chi$ be a character of $K^\times \bs \A_K^\times$ such that $\chi |_{\A^\times} = 1$ and such that $\gcd(C(\chi), d)=1$ where $C(\chi)= N(\cond(\chi))$ .
         Let $f$ be either a holomorphic cuspform of weight $k \ge 2$ or a Maass cuspform of weight $0$ and eigenvalue $\lambda$ with respect to the subgroup $\Gamma_0(dp^n)$ and assume that $f$ is a newform (of trivial nebentypus). Then we have $$L(1/2, f \times \theta_\chi) \ll_{d, p, M, \chi_\infty, \lambda/k, \epsilon} (p^n)^{5/12 + \eps} C(\chi)^{1/2 + \eps}.$$
\end{corollary}
\begin{proof}Let $\pi'$ be the automorphic representation attached to $f$. Note that $\pi'$ is (up to a twist) a Steinberg representation at each prime dividing $d$.   We let $D$ be the indefinite quaternion division algebra of reduced discriminant $d$. Then $\pi'$ transfers to an automorphic representation $\pi \in \mathcal{S}_p$ on $D^\times(\A)$. The corollary now follows immediately from Theorem \ref{t:sub}.
\end{proof}

\subsection{An explicit version of Waldspurger's formula}We now begin the proof of Theorem \ref{t:sub}. We assume the conditions of Theorem \ref{t:sub} for the rest of this Section. Let $\pi = \otimes_v \pi_v \in \mathcal{S}_P$.
Then for \emph{all} finite primes $p$, $\pi_p$ has a $\chi_p$-Waldspurger model; this follows, e.g., from the calculations of Section 5 of \cite{Gross88}. We may further assume that $\pi_\infty$ has a $\chi_\infty$-Waldspurger model, since otherwise the global $\epsilon$-factor $\eps(\pi \times \AI(\chi))$  would equal -1 and we would have $L(1/2, \pi \times \AI(\chi)) =0$, making Theorem \ref{t:sub} trivial.

Since all primes dividing $d$ are inert in $K$, it follows that $K$ embeds in $D$. We fix an embedding $\Phi:K \hookrightarrow D$ and let $T = \Phi(K^\times) \simeq K^\times$ be the corresponding torus inside $G$. We henceforth consider $\chi$ as a character of $\A^\times T(\Q)\bs T(\A^\times)$. Given any $\phi \in V_\pi$, consider the period integral $$P(\phi) = \int_{\A^\times T(\Q)\bs T(\A)}\phi(t)\chi^{-1}(t) dt$$ where $dt$ is the product of local Tamagawa measures. Also, for this Section only, we let the measure on $G(\A)$ be the product of the local Tamagawa measures and define $\langle \phi, \phi \rangle$ with respect to this measure. A beautiful formula of Waldspurger \cite{Waldspurger1985} states that $$\frac{|P(\phi)|^2}{\langle \phi, \phi \rangle} = \zeta(2) \frac{L(1/2, \pi \times \AI(\chi)) }{L(1, \pi, \Ad)} \prod_v \alpha_v(K, \chi, \phi)$$ where the $\alpha_v(K, \chi, \phi)$ are local integrals which equal 1 at almost all places $v$. There have been several papers which have explicitly computed these local integrals at the remaining (ramified) places under certain assumptions, leading to an explicit Waldspurger formula in those cases. We will need such an explicit formula which applies to our setup, due to File, Martin and Pitale \cite{FMP17}.

To state the formula, let us first set up some notation. First of all, we choose the embedding $\Phi:K \hookrightarrow D$ such that $\O_K$ embeds in $\O^\m$ optimally, i.e., $\Phi(K) \cap \O^\m =  \Phi(\O_K)$. Note that for each prime $p$ we have $\Phi(K_p) \cap \O_p^\m =  \Phi(\O_{K,p})$ where $K_p = K \otimes_\Q \Q_p$ and $\O_{K,p} = \O_K \otimes_\Z \Z_p$. Next, we  need to specify the automorphic form $\phi= \otimes_v \phi_v$. For each finite prime $p$ that does not divide $C(\pi)C(\chi)$, we let $\phi_p$ be the (unique up to multiples) non-zero vector in $\pi_p$ that is fixed by $K_p$.

Next, let $p$ be a prime that divides $C(\chi)$ but does not divide $C(\pi)$. Define $m_p$ to be the largest positive integer such that $p^{m_p} | C(\chi)$ and put $c_p = \left \lceil \frac{m_p}{2}
 \right \rceil$. (In fact, $m_p$ is always even, but we won't need this fact). Note that the character $\chi_p$ on $K_p^\times$ is trivial on the subgroup $\Z_p^\times + p^{c_p} \O_{K,p}$.  Now by Section 3 of \cite{Gross88}, there exists a maximal order $R_p$ of $D_p$ such that $R_p \cap \Phi(K_p) = \Z_p + p^{c_p} \Phi(\O_{K,p})$. We let $\phi_p$ be the unique (up to multiples) vector in $\pi_p$ that is fixed by $R_p^\times$. Note that $R_p^\times$ is conjugate to $K_p$; hence $\phi_p$ is a $G_p$-translate of the unique (up to multiples) $K_p$-fixed vector (spherical vector) in $\pi_p$.

Next, let  $p$ be a prime that divides $C(\pi)$. Note that $K_p \simeq \Q_p \oplus \Q_p$. Define $c_p$ as above, and define $n_p = a(\pi_p)$, so that $n_p$ is the largest positive integer such that $p^{n_p} | C(\pi)$. Let $K_0(n_p)$ be as usual the subgroup of $\GL_2(\Z_p)$ consisting of matrices that are upper triangular modulo $p^{n_p}$. Take $g_p \in G_p$ such that $\iota_p(g_p^{-1}T(\Q_p)g_p)$ is the diagonal subgroup of $\GL_2(\Q_p)$. Define the subgroup $K'_0(n_p)$ of $G_p$ via
$$K'_0(n_p) =  g_p \iota_p^{-1}\left(\mat{1}{-p^{-c_p}}{0}{1} K_0(n_p) \mat{1}{p^{-c_p}}{0}{1}\right)g_p^{-1}$$
and let  $\phi_p$ be the unique (up to multiples) vector in $\pi_p$ that is fixed by $K'_0(n_p)$. Note that $\phi_p$ is a $G_p$-translate of the unique (up to multiples) newvector in $\pi_p$.

Finally, we define $\phi_\infty$. Let $K_\infty$ be a maximal compact connected subgroup of $D_\infty$ whose restriction to $T(\R)$ is a maximal compact connected subgroup of $T(\R)$. Let $\phi_\infty$ be a vector of minimal (non-negative) weight such that $\pi_\infty(t_\infty) \phi_\infty = \chi_\infty(t_\infty) \phi_\infty$ for all $t_\infty \in K_\infty \cap T(\R)$.

Put $\phi= \otimes_v \phi_v$. For brevity, put $N=C(\pi)$, $Q=C(\chi)$. Then we have the following explicit version of Waldspurger's formula due to File, Martin, and Pitale (Theorem 1.1 of \cite{FMP17}), simplified to our setting:
\begin{equation}\label{e:waldsexp}\frac{|P(\phi)|^2}{\langle \phi, \phi \rangle} = \frac{C_\infty\zeta(2)}{2\sqrt{Q\disc(K})} \prod_{p|Q} L(1, \rho_{K,p})^2 \prod_{p|N} \left(1+\frac1p \right) \prod_{p|d} \left(1-\frac1p \right)  \frac{L^{Nd}(1/2, \pi \times \AI(\chi)) }{L^{Nd}(1, \pi, \Ad)}. \end{equation}
Above, $L^{Nd}$ denotes the $L$-functions where we omit the Euler factors at primes dividing $Nd$, $\rho_K$ denotes the quadratic character associated to $K/\Q$, and the quantity $C_\infty$ is a positive real number written down explicitly in \cite[7B]{FMP17} that depends only on $\pi_\infty$ and $\chi_\infty$.

\subsection{The proof of Theorem \ref{t:sub}}
We continue to use the notation $N=C(\pi)$, $Q=C(\chi)$. The explicit formula \eqref{e:waldsexp} immediately implies the asymptotic inequality \begin{equation}\label{penult} \left( \frac{\sup_{g \in G(\A)} |\phi(g)|}{\|\phi\|_2} \right)^2 \ge \frac{1}{(\vol(\A^\times T(\Q) \bs T(\A))^2} \frac{|P(\phi)|^2}{\langle \phi, \phi \rangle} \gg_{K, P, \pi_\infty, \chi_\infty, \eps} \frac{1}{\sqrt{Q}} (QN)^{-\eps}  L(1/2, \pi \times \AI(\chi)).\end{equation}

On the other hand, our choice of $\phi$ implies that $\phi_p$ is a translate of the local newvector at all primes $p$. Hence $\phi = R(g) \phi'$ where $g \in G(\A_\f)$ and $(\C\phi', \pi) \in \Aa(G, \mathcal{G})$ with $\mathcal{G}$  as in Proposition \ref{t:mainlocal}. Furthermore $\phi_\infty = \phi'_\infty$ is a vector of weight $k$ where $k$ depends only on $\chi_\infty$. Since the sup-norm does not change under translation, we have, using Corollary \ref{c:mainnewforms} \begin{equation}\label{ult} \left( \frac{\sup_{g \in G(\A)} |\phi(g)|}{\|\phi\|_2} \right)^2  =  \left( \frac{\sup_{g \in G(\A)} |\phi'(g)|}{\|\phi'\|_2} \right)^2  \ll_{\pi_\infty, \chi_\infty, P, \eps} N^{5/12 + \eps}.\end{equation}

Combining \eqref{penult} and \eqref{ult} we obtain $$L(1/2, \pi \times \AI(\chi)) \ll_{K,P,\pi_\infty, \chi_\infty, \eps} N^{5/12+\eps}Q^{1/2+\eps}$$ as desired.

 \bibliography{sup-compact}

\def\cprime{$'$} \def\cprime{$'$} \def\cprime{$'$}
\begin{thebibliography}{10}

\bibitem{assing18}
Edgar Assing.
\newblock On the size of $p-$adic whittaker functions.
\newblock {\em Trans. Amer. Math. Soc.}, 372:5287--5340, 2019.

\bibitem{blomer_2018}
Valentin Blomer.
\newblock Epstein zeta-functions, subconvexity, and the purity conjecture.
\newblock {\em J. Inst. Math. Jussieu}, pages 1--16, 2018.

\bibitem{blomer-harcos-milicevic-maga}
Valentin Blomer, Gergely Harcos, P\'{e}ter Maga, and Djordje Mili\'{c}evi\'{c}.
\newblock The sup-norm problem for {$\rm GL(2)$} over number fields.
\newblock {\em J. Eur. Math. Soc. (JEMS)}, 22(1):1--53, 2020.

\bibitem{blomer-holowinsky}
Valentin Blomer and Roman Holowinsky.
\newblock Bounding sup-norms of cusp forms of large level.
\newblock {\em Invent. Math.}, 179(3):645--681, 2010.

\bibitem{BH06}
Colin~J. Bushnell and Guy Henniart.
\newblock {\em The Local Langlands Conjecture for $\rm GL(2)$}.
\newblock Springer-Verlag, Berlin, 2006.

\bibitem{corbett-saha}
Andrew Corbett and Abhishek Saha.
\newblock On the order of vanishing of newforms at cusps.
\newblock {\em Math. Res. Lett.}, 25(6):1771--1804, 2018.

\bibitem{FMP17}
Daniel File, Kimball Martin, and Ameya Pitale.
\newblock Test vectors and central {$L$}-values for {${\rm GL}(2)$}.
\newblock {\em Algebra Number Theory}, 11(2):253--318, 2017.

\bibitem{Gelbart1975}
Stephen Gelbart.
\newblock {\em Automorphic forms on ad\`ele groups}.
\newblock Princeton University Press, Princeton, N.J.; University of Tokyo
  Press, Tokyo, 1975.
\newblock Annals of Mathematics Studies, No. 83.

\bibitem{Gross88}
Benedict~H. Gross.
\newblock Local orders, root numbers, and modular curves.
\newblock {\em Amer. J. Math.}, 110(6):1153--1182, 1988.

\bibitem{harcos-michel}
Gergely Harcos and Philippe Michel.
\newblock The subconvexity problem for {R}ankin-{S}elberg {$L$}-functions and
  equidistribution of {H}eegner points. {II}.
\newblock {\em Invent. Math.}, 163(3):581--655, 2006.

\bibitem{harcos-templier-1}
Gergely Harcos and Nicolas Templier.
\newblock On the sup-norm of {M}aass cusp forms of large level: {II}.
\newblock {\em Int. Math. Res. Not. IMRN}, 2012(20):4764--4774, 2012.

\bibitem{harcos-templier-2}
Gergely Harcos and Nicolas Templier.
\newblock On the sup-norm of {M}aass cusp forms of large level. {III}.
\newblock {\em Math. Ann.}, 356(1):209--216, 2013.

\bibitem{harris-kudla-1991}
Michael Harris and Stephen~S. Kudla.
\newblock The central critical value of a triple product {$L$}-function.
\newblock {\em Ann. of Math. (2)}, 133(3):605--672, 1991.

\bibitem{hu_triple_2017}
Yueke Hu.
\newblock Triple product formula and the subconvexity bound of triple product
  {$L$}-function in level aspect.
\newblock {\em Amer. J. Math.}, 139(1):215--259, 2017.

\bibitem{Hu:17a}
Yueke Hu.
\newblock Triple product formula and mass equidistribution on modular curves of
  level {$N$}.
\newblock {\em Int. Math. Res. Not. IMRN}, (9):2899--2943, 2018.

\bibitem{HN-minimal}
Yueke Hu and Paul~D. Nelson.
\newblock New test vector for {W}aldspurger's period integral.
\newblock {\em ArXiv e-prints}, 2018.
\newblock arXiv:1810.11564v1.

\bibitem{HNS}
Yueke Hu, Paul~D. Nelson, and Abhishek Saha.
\newblock Some analytic aspects of automorphic forms on {$\rm GL(2)$} of
  minimal type.
\newblock {\em Comment. Math. Helv.}, 94(4):767--801, 2019.

\bibitem{humphries-khan}
Peter Humphries and Rizwanur Khan.
\newblock On the random wave conjecture for dihedral {M}aa\ss forms.
\newblock {\em Geom. Funct. Anal.}, 30(1):34--125, 2020.

\bibitem{MR2061214}
Henryk Iwaniec and Emmanuel Kowalski.
\newblock {\em Analytic number theory}, volume~53 of {\em American Mathematical
  Society Colloquium Publications}.
\newblock American Mathematical Society, Providence, RI, 2004.

\bibitem{iwan-sar-85}
Henryk Iwaniec and Peter Sarnak.
\newblock {$L^\infty$} norms of eigenfunctions of arithmetic surfaces.
\newblock {\em Ann. of Math. (2)}, 141(2):301--320, 1995.

\bibitem{KMV02}
E.~Kowalski, P.~Michel, and J.~VanderKam.
\newblock Rankin-{S}elberg {$L$}-functions in the level aspect.
\newblock {\em Duke Math. J.}, 114(1):123--191, 2002.

\bibitem{marsh15}
Simon Marshall.
\newblock Local bounds for {$L^p$} norms of {M}aass forms in the level aspect.
\newblock {\em Algebra \& Number Theory}, 10(4):803--812, 2016.

\bibitem{michel-2009}
Philippe Michel and Akshay Venkatesh.
\newblock The subconvexity problem for {${\rm GL}_2$}.
\newblock {\em Publ. Math. Inst. Hautes \'Etudes Sci.}, (111):171--271, 2010.

\bibitem{asbjorn}
Asbjorn {Nordentoft}.
\newblock {Hybrid subconvexity for class group $L$-functions and uniform sup
  norm bounds of Eisenstein series}.
\newblock {\em arXiv e-prints}, Mar 2019.
\newblock arXiv:1903.03932.

\bibitem{Oh02}
Hee Oh.
\newblock Uniform pointwise bounds for matrix coefficients of unitary
  representations and applications to {K}azhdan constants.
\newblock {\em Duke Math. J.}, 113(1):133--192, 2002.

\bibitem{Rohrlich1994}
David Rohrlich.
\newblock Elliptic curves and the {W}eil-{D}eligne group.
\newblock In {\em Elliptic curves and related topics}, volume~4 of {\em CRM
  Proc. Lecture Notes}, pages 125--157. Amer. Math. Soc., Providence, RI, 1994.

\bibitem{sahasupwhittaker}
Abhishek Saha.
\newblock Large values of newforms on {${\rm GL}(2)$} with highly ramified
  central character.
\newblock {\em Int. Math. Res. Not. IMRN}, (13):4103--4131, 2016.

\bibitem{saha-sup-level-hybrid}
Abhishek Saha.
\newblock Hybrid sup-norm bounds for {M}aass newforms of powerful level.
\newblock {\em Algebra \& Number Theory}, 11(5):1009--1045, 2017.

\bibitem{saha-sup-minimal}
Abhishek Saha.
\newblock Sup-norms of eigenfunctions in the level aspect for compact
  arithmetic surfaces.
\newblock {\em Math. Ann.}, 376(1-2):609--644, 2020.

\bibitem{sarnak93}
Peter Sarnak.
\newblock Arithmetic quantum chaos.
\newblock {\em Blyth Lectures. Toronto}, 1993.
\newblock Available at
  \verb"http://publications.ias.edu/sites/default/files/Arithmetic%20Quantum%20Chaos.pdf".

\bibitem{sawin19}
Will {Sawin}.
\newblock {A geometric approach to the sup-norm problem for automorphic forms:
  the case of newforms on $GL_2(\mathbb F_q(T))$ with squarefree level}.
\newblock {\em arXiv e-prints}, page arXiv:1907.08098, Jul 2019.

\bibitem{shalika-tanaka}
J.~A. Shalika and S.~Tanaka.
\newblock On an explicit construction of a certain class of automorphic forms.
\newblock {\em American Journal of Mathematics}, 91(4):1049--1076, 1969.

\bibitem{templier-sup}
Nicolas Templier.
\newblock On the sup-norm of {M}aass cusp forms of large level.
\newblock {\em Selecta Math. (N.S.)}, 16(3):501--531, 2010.

\bibitem{tunnell78}
Jerrold~B. Tunnell.
\newblock On the local {L}anglands conjecture for {$GL(2)$}.
\newblock {\em Invent. Math.}, 46(2):179--200, 1978.

\bibitem{Waldspurger1985}
Jean-Loup Waldspurger.
\newblock Sur les valeurs de certaines fonctions {$L$} automorphes en leur
  centre de sym\'etrie.
\newblock {\em Compositio Math.}, 54(2):173--242, 1985.

\bibitem{Wu}
Han {Wu} and Nickolas {Andersen}.
\newblock {Explicit subconvexity for $\mathrm{GL}_2$ and some applications}.
\newblock {\em arXiv e-prints}, Dec 2018.
\newblock arXiv:1812.04391.

\bibitem{young18}
Matthew~P. Young.
\newblock A note on the sup norm of {E}isenstein series.
\newblock {\em Q. J. Math.}, 69(4):1151--1161, 2018.

\end{thebibliography}

\end{document}